\documentclass[a4paper,12pt]{article}

\usepackage{mymacros}

\usepackage{curves}

\theoremstyle{plain}

\newtheorem{prop}[theorem]{Proposition}

\theoremstyle{definition}
\newtheorem{defn}[theorem]{Definition}

\newcommand{\scYv}{{\check{\scY}}}

\newcommand{\cYv}{{\check{\scY}}}
\newcommand{\sm}{\mathrm{sm}}
\newcommand{\HW}{\mathop{HW}\nolimits}
\newcommand{\HF}{\mathop{HF}\nolimits}

\newcommand{\ord}{\operatorname{ord}}

\title{Dual torus fibrations and \\
homological mirror symmetry for $A_n$-singularities}
\author{Kwokwai Chan and Kazushi Ueda}
\date{}

\begin{document}

\maketitle

\begin{abstract}
We study homological mirror symmetry
%(\'a la Kontsevich)
for not necessarily compactly supported coherent sheaves on
%(the complement of the anti-canonical divisor of)
the minimal resolutions of $A_n$-singularities.
An emphasis is put on the relation with the Strominger-Yau-Zaslow conjecture.
\end{abstract}

\section{Introduction}

Let $Y$ be the affine hypersurface
\begin{equation} \label{eq:Y}
 Y=\lc (z, u, v)\in \bCx \times \bC^2 \relmid u v= z^{-1} f(z) \rc,
\end{equation}
where
$
 f(z) = (z - a_0) (z - a_1) \cdots (z - a_n)
$
is a polynomial of degree $n+1$ with
%For convenience, we assume that the polynomial $f$ has
mutually distinct positive real zeros
$
 0 < a_0 < a_1 < \dots < a_n.
$
We equip $Y$ with the symplectic form
\begin{equation*}
 \omega = - \frac{\sqrt{-1}}{2}
  \left. \lb
   \frac{dz \wedge d\bar{z}}{|z|^2}
    + du \wedge d\ubar + dv \wedge d\vbar
  \rb \right|_Y.
\end{equation*}
The projection
\begin{align} \label{eq:Y-conic}
 \pi : Y \to \bCx, \quad
 (z, u, v) \mapsto z
\end{align}
is a conic fibration whose discriminant is given by the zeros
$
 \Delta = \{ a_0, \dots, a_n \}
$
of the polynomial $f$.
Using this, one can show that the map
\begin{align}
%\begin{array}{cccc}
% \rho : & Y & \to & B \\
%  & \vin & & \vin \\
%  & (z, u, v) & \mapsto & \lb \log |z|,
%   \frac{1}{2}(|u|^2 - |v|^2) \rb
%\end{array}
 \rho : Y \to B, \quad
  (z, u, v) \mapsto \lb \log |z|,
   \frac{1}{2}(|u|^2 - |v|^2) \rb
      \label{eq:Y-SYZ}
\end{align}
is a Lagrangian torus fibration over the base $B = \bR^2$, whose discriminant locus is given by
$$
 \Gamma := \{ (s_0, 0), (s_1, 0), \dots, (s_n, 0) \},
$$
where
$
 s_i = \log a_i
$
for $i = 0, \dots, n$. Each fiber $L$ of this Lagrangian torus fibration is {\em special} in the sense that one has
$$
 \left. \Im \lb e^{\sqrt{-1} \theta} \Omega \rb \right|_L = 0
$$
for some $\theta \in \bR$, where
\begin{align} \label{eq:Y-Omega}
 \Omega = \Res \frac{d z \wedge d u \wedge d v}{z u v - f(z)}
%  = \frac{d z \wedge d u}{z u}
  = d \log z \wedge  d \log u
\end{align}
is a nowhere-vanishing holomorphic 2-form on Y.

Strominger, Yau, and Zaslow
\cite{Strominger-Yau-Zaslow}
conjectured that
any Calabi-Yau manifold admits a special Lagrangian torus fibration,
and its mirror is obtained as the dual torus fibration.
In this paper,
we apply their ideas on the fibration \eqref{eq:Y-SYZ},
which we will refer to as the {\em SYZ fibration}.

Given the SYZ fibration $\rho : Y \to B$,
one can equip the complement $B^\sm := B \setminus \Gamma$
of the discriminant
with two tropical affine structures.
One is called the {\em symplectic affine structure},
and the other is called the {\em complex affine structure}.
Here a manifold with a {\em tropical affine structure} is a manifold
which is obtained by gluing open subsets of $\bR^m$
by the action of the affine linear group
$\bR^m \rtimes \GL_m(\bZ)$
(or $\bR^m \rtimes \SL_m(\bZ)$ if the manifold is oriented).
An {\em integral affine manifold} is the special case
when the gluing maps belong to $\bZ^m \rtimes \GL_m(\bZ)$.
The symplectic affine structure is defined by
first taking a basis $\{ \gamma_1, \gamma_2 \}$
of the space of local sections
of the relative homology bundle
$(R^1 \rho_* \bZ)^\vee$
and integrating the symplectic form $\omega$
along these cycles to obtain 1-forms on $B$;
$
 d x_i = \int_{\gamma_i} \omega.
$
Local affine coordinate on $B^\sm$
are primitives of these 1-forms.
The complex affine structure is defined similarly
by using $\Im (e^{\sqrt{-1} \theta} \Omega)$
instead of $\omega$,
and is {\em Legendre dual}
to the symplectic affine structure
\cite{MR1876068}.

Following earlier works
(cf.~e.g.~\cite{Fukaya_MMTAAMS,
MR2181810,
Gross-Siebert_RAGCG,
Auroux_MSTD,
Auroux_SLFWCMS,
CLL11,
Abouzaid-Auroux-Katzarkov_LFB}
and references therein),
the mirror $\Yv$ of $Y$ is identified
in \cite{Chan_HMSARTD}
with the complement of an anti-canonical divisor
in the minimal resolution
of the $A_n$-singularity.
This mirror $\Yv$ admits a special Lagrangian torus fibration,
which is an {\em SYZ mirror}
in the sense that the symplectic and complex affine structures
%associated with the fibration
are interchanged
between $Y$ and $\Yv$.

Let $E_1, \ldots, E_n \subset \Yv$
be the irreducible components
of the exceptional divisor
in the minimal resolution.
Then there is an isomorphism
\begin{align} \label{eq:PicYv}
\begin{array}{cccc}
 \deg : & \Pic \Yv & \to & \bZ^n \\
  & \vin & & \vin \\
  & \scL & \mapsto & (\deg \scL|_{E_i})_{i=1}^n
\end{array}
\end{align}
of abelian groups.
We write the line bundle associated with $\bsd \in \bZ^n$ as $\scL_\bsd$.

Given an SYZ fibration,
it is expected that Lagrangian sections of the original manifold
and holomorphic line bundles on the mirror manifold
are related by a kind of Fourier transform
\cite{Arinkin-Polishchuk_FCFT, Leung-Yau-Zaslow_SLHYM},
which we refer to as the {\em SYZ transformation}.
We introduce the notion of a {\em strongly admissible path}
in $\bCx \setminus \Delta$,
and associate an exact Lagrangian section
$L_\gamma \subset Y$
of the SYZ fibration \eqref{eq:Y-SYZ}
to each strongly admissible path.
The {\em winding number}
$
 \bsw(\gamma) = (w_1(\gamma), \dots, w_n(\gamma)) \in \bZ^n
$
of a strongly admissible path is defined
as the intersection numbers
with the closed intervals $[a_{i-1}, a_i]$ for $i = 1, \dots, n$.

The first main result in this paper
is the following:
%(see Theorem \ref{th:SYZ_transform} in Section \ref{sec:SYZ_transform}):

\begin{theorem} \label{th:main1}
For a strongly admissible path $\gamma$, the SYZ transformation of the Lagrangian $L_\gamma$ is given by the line bundle $\scL_{-\bsw(\gamma)}$.
\end{theorem}

Next we consider another symplectic manifold defined by
\begin{align} \label{eq:Y'1}
 Y' = \lc (z, u, v) \in \bCx \times \bC^2 \relmid u v = \frac{1}{z} + z^n \rc.
\end{align}
Note that $Y'$ is related to $Y$
by moving $a_i$ to the roots of unity,
and hence is symplectomorphic to it.
The map \eqref{eq:Y-SYZ} gives
a special Lagrangian torus fibration on $Y'$,
whose discriminant consists only of the origin.
%The monodromy around the origin
%for this Lagrangian torus fibration
%is the same as the monodromy for $Y$
%along a large circle
%encircling all the discriminant locus.
%$
%\begin{pmatrix}
% 1 & n+1 \\
% 0 & 1
%\end{pmatrix},
%$
%which is

The mirror $\scYv$ for $Y'$ is the smooth stack
obtained by removing an anti-canonical divisor
from the total space $\scK$ of the canonical bundle of the weighted projective line $\bP(1, n)$.
%In the language of \cite[Section 10.1]{Abouzaid-Auroux-Katzarkov_LFB},
%one takes the polynomial $f(z) = 1/z + T + z^n$ instead,
%so that the polytope
%$
% \Delta_Y = \{ (\xi, \eta) \in \bR^2 \mid \eta \ge \varphi(\xi) \}
%$
%defined by the corresponding tropical polynomial
%$$
% \varphi(\xi) = \max \{ - \xi, 1, n \xi \}
%$$
%gives the desired toric stack.
Although the McKay correspondence
\cite{Kapranov-Vasserot}
gives a derived equivalence
\begin{equation} \label{eq:d-equiv}
 D^b \coh \Yv \cong D^b \coh \scYv,
\end{equation}
the Picard groups are {\em not} isomorphic:
$$
 \Pic \scYv \cong \Pic \bP(n, 1) \cong \bZ
  \not \cong \bZ^n \cong \Pic \Yv.
$$
Let $\scO_{\scYv}(i)$ be the line bundle on $\scYv$
obtained by restricting the pull-back of $\scO_{\bP(1,n)}(i)$.
%as $\scO_{\scYv}(i)$.
%(or just $\scO(i)$ when there is no danger of confusion).

Homological mirror symmetry
\cite{Kontsevich_HAMS, Kontsevich_ENS98}
for $\bP(1,n)$ gives an equivalence
\begin{align} \label{eq:wpl-hms}
 D^b \coh \bP(1,n) \cong D^b \Fuk W,
\end{align}
where $\Fuk W$ is a Fukaya category
associated with the Lefschetz fibration
\begin{align*}
\begin{array}{cccc}
 W : & \bCx & \to & \bC \\
  & \vin & & \vin \\
  & z & \mapsto & 1/z + z^n.
\end{array}
\end{align*}
This is a special case of
(a generalization to toric stacks of)
the work of Abouzaid
\cite{Abouzaid_HCRMSTV,Abouzaid_MHTGHMSTV}
on homological mirror symmetry
for toric varieties.
There are also works by Fang \cite{Fang_HMSTP}
and Fang, Liu, Treumann and Zaslow
\cite{Fang-Liu-Treumann-Zaslow_TD,
Fang-Liu-Treumann-Zaslow_CCC}
on homological mirror symmetry for toric varieties,
which are also motivated by the Strominger-Yau-Zaslow conjecture
but different from the work of Abouzaid.

Under the equivalence \eqref{eq:wpl-hms},
the collection
$( \scO_{\bP(1,n)}(i) )_{i=0}^n$
of line bundles
is mapped to Lefschetz thimbles
$(\Delta_i)_{i=0}^n$.
These Lefschetz thimbles can be lifted
to Lagrangian sections $(L_i)_{i=0}^n$
of the SYZ fibration for $Y'$.
Let $\scW'$ be the full subcategory
of the {\em wrapped Fukaya category} of $Y'$
consisting of $(L_i)_{i=0}^n$.
%Here the choice for the wrapping Hamiltonian follows
%that of Pascaleff \cite{Pascaleff_FCMPP}.

The second main result in this paper is the following:

\begin{theorem} \label{th:main2}
There is an equivalence
\begin{align} \label{eq:main2}
 D^b \scW' \cong D^b \coh \scYv
\end{align}
of triangulated categories sending $L_i$ to $\scO_{\scYv}(i)$ for $i = 0, \ldots, n$.
\end{theorem}

The proof is based on
%homological mirror symmetry for $\bP(1,n)$ and
an analysis of the behavior of the wrapped Fukaya category under suspension,
and depends heavily on the work of Pascaleff \cite{Pascaleff_FCMPP}.
We expect that $(L_i)_{i=0}^n$ generates
the wrapped Fukaya category,
so that the left hand side of \eqref{eq:main2} is
the whole wrapped Fukaya category.
%There is also another choice for the wrapping Hamiltonian
%such that the resulting wrapped Fukaya category is equivalent to $D^b \coh \scK$.

Theorems \ref{th:main1} and \ref{th:main2}
are compatible in the following sense:
There exists a symplectomorphism $Y \simto Y'$ which induces an equivalence
$$
D^b \scW \simto D^b \scW',
$$
so that the Lagrangians $(L_i)_{i=0}^n$ in $Y'$ are images of
Lagrangians $(L_{\gamma_i})_{i=0}^n$ in $Y$ associated with
certain strongly admissible paths $\gamma_0,\gamma_1,\ldots,\gamma_n$
in $\bCx \setminus \Delta$. One can then choose a derived equivalence
$$
D^b \coh \scYv \simto D^b \coh \Yv
$$
in such a way that the images of $(L_{\gamma_i})_{i=0}^n$ under the composition
$$
D^b \scW \simto D^b \scW' \simto D^b \coh \scYv \simto D^b \coh \Yv
$$
of equivalences are precisely given by their SYZ transformations
described in Theorem \ref{th:main1}.
%More details can be found in Section \ref{sc:An}.
This shows that homological mirror symmetry for non-compact branes
is realized by SYZ
just as for compact branes \cite{Chan_HMSARTD}.

This paper is organized as follows:
In \pref{sec:SYZ},
we briefly recall the SYZ construction of the mirror manifold
$\Yv$
%from $(Y,\omega)$
from \cite[Section 3]{Chan_HMSARTD}.
In \pref{sec:Lagr_submfd_over_paths},
we introduce the notion of a strongly admissible path to which we
associate a noncompact Lagrangian submanifold in $(Y,\omega)$,
which is a section of the SYZ fibration $\rho$.
In \pref{sec:SYZ_transform},
we describe holomorphic line bundles on $\Yv$
obtained as SYZ transformations
of noncompact Lagrangian submanifolds
associated with strongly admissible paths
and prove \pref{th:main1}.
The proof of \pref{th:main2} is given
in \pref{sc:hms}.
%we outline a strategy
%to deduce homological mirror symmetry for
%%(the complement of an anti-canonical divisor of)
%the canonical bundle
%from that for the toric Fano stack,
%and prove \pref{th:main2}.
%The proof of \pref{th:main2} is given in \pref{sc:An}.
%Although the proof itself is logically independent of \pref{sc:conjectures},
%it would be helpful to think of it as a specific example
%of a more general construction
%outlined in \pref{sc:conjectures}.

{\em Acknowledgment:}
K.~C. is supported by Hong Kong RGC Direct Grant for Research 2011/2012
(Project ID: CUHK2060434).
K.~U. is supported by JSPS Grant-in-Aid for Young Scientists No.24740043.
A part of this paper was worked out
while K.~U. was visiting the Chinese University of Hong Kong,
whose hospitality
%and nice research environment
is gratefully acknowledged.

\section{SYZ mirror symmetry}
 \label{sec:SYZ}

%In Section \ref{sec:SYZ},
%we briefly review the construction of the mirror manifold
%$\check{Y}$ from $(Y,\omega)$ using SYZ.

We start with the Hamiltonian $S^1$-action on $(Y,\omega)$:
\begin{equation*}
e^{2\pi\sqrt{-1}t}\cdot(u,v,z) = \left(e^{2\pi\sqrt{-1}t}u,e^{-2\pi\sqrt{-1}t}v,z\right),
\end{equation*}
whose moment map is given by
\begin{equation*}
\mu(u,v,z)=\frac{1}{2}\left(|u|^2-|v|^2\right).
\end{equation*}
Let $B:=\bR^2$. Then the map $\rho:Y\to B$ defined by
\begin{equation*}
\rho(u,v,z)=(\log|z|,\mu)=\left(\log|z|,\frac{1}{2}\left(|u|^2-|v|^2\right)\right)
\end{equation*}
is a Lagrangian torus fibration on $Y$, whose discriminant locus is given by the finite set
\begin{equation*}
\Gamma:=\{(s_0,0), (s_1,0), \ldots, (s_n,0)\}\subset B,
\end{equation*}
where we denote $s_i := \log a_i$ for $i=0,\ldots,n$.
This is usually called an \emph{SYZ fibration} of $(Y,\omega)$.

Let $B^\sm:=B\setminus\Gamma$ be the smooth locus.
Then the fiber $T_{s,\lambda}$ over a point $(s,\lambda)\in B^\sm$ is a smooth Lagrangian torus in $Y$, and each of the fiber $T_{s_i,0}$ over $(s_i,0)\in\Gamma$ is singular with one nodal singularity. Furthermore, the locus of Lagrangian torus fibers which bound nonconstant holomorphic disks is given by a union of vertical lines:
\begin{equation*}
H:=\{(s,\lambda)\in B \mid s=s_i\textrm{ for some $i=0,1,\ldots,n$}\}.
\end{equation*}
Each connected component of $H$ is called a \emph{wall} in $B$.

The Lagrangian torus fibration $\rho:Y\to B$ induces an integral affine structure on $B^\sm$, called the \textit{symplectic affine structure}. As we have mentioned, integrality of the affine structure means that the transition maps between charts on $B^\sm$ are elements of the integral affine linear group $$\textrm{Aff}(\bZ^2):=\textrm{GL}_2(\bZ)\rtimes\bZ^2.$$
This ensures that $\Lambda\subset TB^\sm$, the family of lattices locally generated by $\partial/\partial x_1,\partial/\partial x_2$ where $x_1,x_2$ are local affine coordinates on $B^\sm$, is well-defined.

The Strominger-Yau-Zaslow (SYZ) conjecture \cite{Strominger-Yau-Zaslow} suggested that a mirror of $(Y,\omega)$ can be constructed by fiberwise-dualizing an SYZ fibration on $(Y,\omega)$. More precisely, one defines the \emph{semi-flat mirror} of $(Y,\omega)$ as the moduli space $\check{Y}_0$ of pairs $(T_{s,\lambda},\nabla)$ where $\nabla$ is a flat $U(1)$-connection (up to gauge equivalence) on the trivial line bundle $\underline{\bC}:=\bC\times T_{s,\lambda}$ over a smooth Lagrangian torus fiber $T_{s,\lambda}$. Topologically, $\check{Y}_0$ is simply the quotient $TB^\sm/\Lambda$ of the tangent bundle $TB^\sm$ by $\Lambda$. This is naturally a complex manifold where the local complex coordinates are given by exponentiations of complexifications of local affine coordinates on $B^\sm$ (while the quotient $T^*B^\sm/\Lambda^\vee$ of the cotangent bundle $T^*B^\sm$ by the dual lattice $\Lambda^\vee$, which is canonically a symplectic manifold, is contained in $Y$ as an open dense subset such that the restriction of $\omega$ to it gives the canonical symplectic structure). However, this is not quite the correct mirror manifold because the natural complex structure on $\check{Y}_0$ is not globally defined due to nontrivial monodromy of the affine structure around each singular point $(s_i,0)\in\Gamma$.

A more concrete description of this phenomenon is as follows. For $i=0,1,\ldots,n$, consider the strip $B_i:=(s_{i-1},s_{i+1})\times\bR$ (where we set $s_{-1}:=-\infty$ and $s_{n+1}:=+\infty$). A covering of $B^\sm$ is given by the open sets
\begin{equation*}
U_i:=B_i\setminus[s_i,s_{i+1})\times\{0\},\quad V_i:=B_i\setminus(s_{i-1},s_{i}]\times\{0\}.
\end{equation*}
$i=0,1,\ldots,n$. The intersection $U_i\cap V_i$ consists of two connected components:
$$U_i\cap V_i=B_i^+\cup B_i^-,$$
where $B_i^+$ (resp. $B_i^-$) corresponds to the component in which $\lambda>0$ (resp. $\lambda<0$).

\begin{figure}[ht]
\setlength{\unitlength}{1mm}
\centering
\begin{picture}(100,75)
\put(10,2){\vector(0,1){69}} \put(9,72){$\lambda$} \put(-5,35){\vector(1,0){110}} \put(106,34){$s$} \put(8,32){$0$}
\curve(25,70, 25,62) \curve(25,61, 25,53) \curve(25,52, 25,44) \curve(25,43, 25,35) \curve(25,34, 25,26) \curve(25,25, 25,17) \curve(25,16, 25,8) \curve(25,7, 25,4) \put(23.4,34){$\times$} \put(18.5,32.5){$s_{i-1}$}
\curve(45,70, 45,62) \curve(45,61, 45,53) \curve(45,52, 45,44) \curve(45,43, 45,35) \curve(45,34, 45,26) \curve(45,25, 45,17) \curve(45,16, 45,8) \curve(45,7, 45,4) \put(43.4,34){$\times$} \put(42,32.5){$s_i$}
\curve(60,70, 60,62) \curve(60,61, 60,53) \curve(60,52, 60,44) \curve(60,43, 60,35) \curve(60,34, 60,26) \curve(60,25, 60,17) \curve(60,16, 60,8) \curve(60,7, 60,4) \put(58.4,34){$\times$} \put(60.5,32.5){$s_{i+1}$}
\curve(80,70, 80,62) \curve(80,61, 80,53) \curve(80,52, 80,44) \curve(80,43, 80,35) \curve(80,34, 80,26) \curve(80,25, 80,17) \curve(80,16, 80,8) \curve(80,7, 80,4) \put(78.4,34){$\times$} \put(80.5,32.5){$s_{i+2}$}
\put(43,57){\vector(0,1){5}} \put(41.5,63){$w$} \put(43,57){\vector(-1,0){5}} \put(34,56.5){$u_i$}
\put(47,57){\vector(0,1){5}} \put(45.5,63){$w$} \put(47,57){\vector(1,0){5}} \put(52.5,56.5){$v_{i+1}$}
\end{picture}
\caption{The base affine manifold $B$.}\label{base_B}
\end{figure}
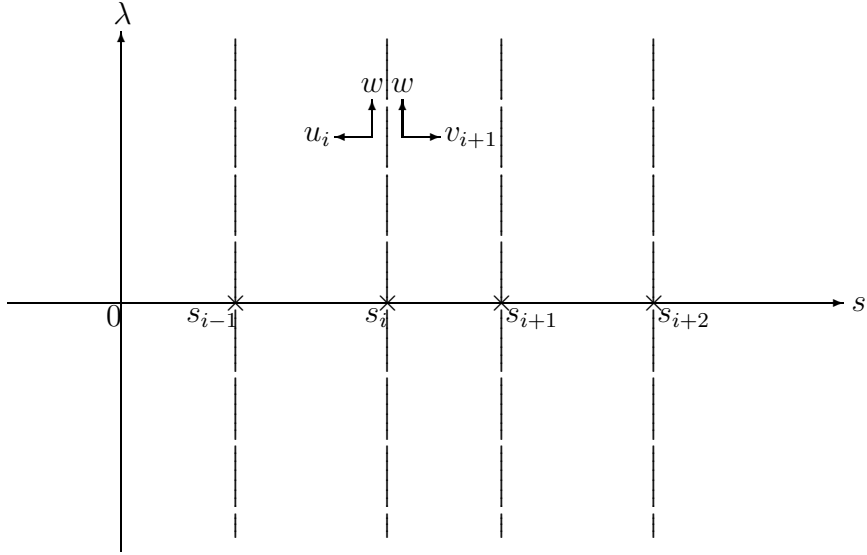

On $U_i$ (resp. $V_i$), denote by $u_i$ (resp. $v_{i+1}$) the exponentiation of the complexification of the left-pointing (resp. right-pointing) affine coordinate. (Our convention is that for a real number $x\in\bR$, its complexification is given by $-x+\sqrt{-1}y$.) Also, denote by $w$ the exponentiation of the complexification of the upward-pointing affine coordinate. See Figure \ref{base_B}.

Now $w$ is a globally defined coordinate. Geometrically, it can be written as
\begin{equation*}
w(T_{s,\lambda},\nabla)=\left\{\begin{array}{ll}
                        \exp\left(-\int_\alpha\omega\right)\textrm{hol}_\nabla(\partial\alpha) & \textrm{ for $\lambda\geq0$},\\
                        \exp\left(\int_\alpha\omega\right)\textrm{hol}_\nabla(-\partial\alpha) & \textrm{ for $\lambda\leq0$},
                        \end{array}\right.
\end{equation*}
where $\alpha\in\pi_2(Y,T_{s_i,\lambda})$ is the class of the holomorphic disk in $Y$ bounded by $T_{s_i,\lambda}$ for some $i=0,1,\ldots,n$, and $\textrm{hol}_\nabla(\partial\alpha)$ is the holonomy of the flat $U(1)$-connection $\nabla$ around the boundary $\partial\alpha\in\pi_1(T_{s_i,\lambda})$. The class $\alpha$ changes to $-\alpha$ when one moves from $\lambda>0$ to $\lambda<0$.

But the other coordinates $u_i$ and $v_{i+1}$ are not globally defined:
Since the monodromy of the affine structure going counter-clockwise
around $(s_i,0)\in\Gamma$ is given by the matrix
\begin{equation*}
\left(\begin{array}{ll}
        1 & 1\\
        0 & 1
      \end{array}\right),
\end{equation*}
the coordinates $(u_i,w)$ and $(v_{i+1},w)$ glue on $U_i\cap V_i$ according to
\begin{eqnarray*}
u_i=v_{i+1}^{-1} & \textrm{ on $B_i^+$,}\\
u_i=v_{i+1}^{-1}w & \textrm{ on $B_i^-$.}
\end{eqnarray*}
Hence, the monodromy of the complex coordinates
going counter-clockwise around the point $(s_i,0)$ is nontrivial
and given by
\begin{equation*}
u_i\mapsto u_iw,\quad v_{i+1}^{-1}\mapsto v_{i+1}^{-1}w.
\end{equation*}
In particular, these complex coordinates on $TB_i/\Lambda$
do not form a globally defined complex structure on $\check{Y}_0$.

The examples in \cite{Auroux_MSTD, Auroux_SLFWCMS, CLL11, Abouzaid-Auroux-Katzarkov_LFB}
suggest the following construction
(we should mention that these are all special cases
of the constructions in Kontsevich-Soibelman \cite{Kontsevich-Soibelman_NAANCG}
and Gross-Siebert \cite{Gross-Siebert_RAGCG},
but those general constructions largely ignore the symplectic aspects
and hence Lagrangian torus fibration structures -
they build the mirror family directly from tropical data in the base $B$):
in order to get the genuine mirror manifold, we need to modify
the gluing of complex charts on $\check{Y}_0$ by quantum corrections
from disk instantons
\begin{eqnarray}\label{wc+}
u_i=v_{i+1}^{-1}+v_{i+1}^{-1}w =& v_{i+1}^{-1}(1+w) & \textrm{ on $B_i^+$,}\\
u_i=v_{i+1}^{-1}w+v_{i+1}^{-1} =& v_{i+1}^{-1}w(1+w^{-1}) & \textrm{ on $B_i^-$.}\label{wc-}
\end{eqnarray}
Geometrically, the term $v_{i+1}^{-1}w$ in the first formula \eqref{wc+} should be viewed as multiplying $v_{i+1}^{-1}$ by $w$ where $w$ corresponds to the unique nonconstant holomorphic disk bounded by the Lagrangian torus $T_{0,\lambda}$ whose area is given by
$$\lambda=-\log|w|>0.$$
This means that we are correcting the term $v_{i+1}^{-1}$ by adding the contribution from the holomorphic disk bounded by $T_{0,\lambda}$ when we cross the upper half of the wall $\{s_i\}\times\bR\subset H$. In the same way, the term $v_{i+1}^{-1}$ in the second formula \eqref{wc-} is given by multiplying $v_{i+1}^{-1}w$ by $w^{-1}$ where $w^{-1}$ now corresponds to the unique nonconstant holomorphic disk bounded by the Lagrangian torus $T_{0,\lambda}$ whose area is equal to
$$-\lambda=\log|w|=-\log|w^{-1}|>0.$$
So we are correcting $v_{i+1}^{-1}w$ by the disk bounded by $T_{0,\lambda}$ when we cross the lower half of the wall $\{s_i\}\times\bR\subset H$.

The formulas \eqref{wc+}, \eqref{wc-} can actually be interpreted as wall-crossing formulas for the counting of Maslov index two holomorphic disks in $Y$ bounded by Lagrangian torus fibers. To see this, we need to partially compactify $Y$ by allowing $z$ to take values in $\bC$ and replacing $\rho$ by the map
$$(u,v,z) \mapsto \left(|z|,\mu\right) = \left(|z|,\frac{1}{2}\left(|u|^2-|v|^2\right)\right).$$
Then the base becomes an affine manifold with both singularities and boundary: it is a right half-space $\bar{B}$ in $\bR^2$, and the boundary $\partial\bar{B}$ corresponds to the divisor in $Y$ defined by $z=0$. In this situation, each of the local coordinates $u_i$, $v_{i+1}^{-1}$ and $v_{i+1}^{-1}w$ can be expressed as in the form
\begin{equation*}
\exp\left(-\int_\beta\omega\right)\textrm{hol}_\nabla(\partial\beta),
\end{equation*}
for a suitable relative homotopy class $\beta\in\pi_2(Y,T_{s,\lambda})$ of Maslov index two
(cf. Auroux \cite{Auroux_MSTD, Auroux_SLFWCMS}).

In any case, the modification of gluing cancels the monodromy and defines global complex coordinates $u_i,v_i,w$ on $TB_i/\Lambda$ (topologically, $TB_i/\Lambda\cong(\bC^\times)^2$) related by
\begin{equation}\label{gluing}
u_i=v_i^{-1},\quad u_{i+1}=v_{i+1}^{-1}\textrm{ and }u_iv_{i+1}=1+w.
\end{equation}
The instanton-corrected mirror $\check{Y}$ is then obtained by gluing together the pieces $TB_i/\Lambda$ ($i=0,1,\ldots,n$) according to \eqref{gluing}. On the intersection $(TB_{i-1}/\Lambda)\cap(TB_i/\Lambda)$, we have
\begin{equation*}
u_i=v_i^{-1},\quad u_iv_{i+1}=1+w=u_{i-1}v_i.
\end{equation*}
We remark that, by our definition, the mirror manifold $\check{Y}$ will have ``gaps" because for instance $u,v,w$ are $\bC^\times$-valued. A natural way to ``complete" the mirror manifold and fill out its gaps is by rescaling the symplectic structure of $Y$ and performing analytic continuation; see \cite[Section 4.2]{Auroux_MSTD} for a discussion of this ``renormalization" process.

Finally, one observes that this is precisely the gluing of complex charts in the toric resolution $X\to\bC^2/\bZ_{n+1}$ of the $A_n$-singularity. We thus conclude that the mirror is precisely given by the complex manifold
\begin{equation*}
\check{Y}:=X\setminus D,
\end{equation*}
where $X=X_\Sigma$ is the toric surface defined by the 2-dimensional fan
$\Sigma\in\bR^2$ generated by
$$
 \{ v_i:=(i-1,1)\in N \mid i=0,1,\ldots,n+1 \},
$$
and $D$ is the hypersurface in $X_\Sigma$
defined by $h:=\chi^{(0,1)}=1$.
Here $h:X\to\bC$ is the holomorphic function
whose zero locus is the union of all the toric divisors in $X$.
For $i=1,\ldots,n$,
if we let $S_i$ denote the interval $(s_{i-1},s_i)\times\{0\}\subset B$,
then (the closure of) the quotient $TS_i/TS_i\cap\Lambda$ in $\check{Y}$
is one of the $n$ exceptional divisors in $X_\Sigma$,
each of which is a $(-2)$-curve.
We denote this exceptional divisor by $E_i\subset\check{Y}$.

Let $M = \Hom(N, \bZ)$ be the lattice of characters
of the dense torus of $X_\Sigma$ and
$
 \Sigma(1) = \{ v_0, \ldots, v_{n+1} \}
$
be the set of one-dimensional cones of the fan $\Sigma$.
If we write the toric divisor on $X_\Sigma$
corresponding to $v_i \in \Sigma(1)$ as $D_i$,
then we have $D_i = E_i$ for $i = 1, \ldots, n$, and
$D_0$ and $D_{n+1}$ are non-compact divisors.
One can easily show
from the divisor sequence
$$
 0 \to M \to \bZ^{\Sigma(1)} \to \Pic X_\Sigma \to 0
$$
that $\Pic X_\Sigma$ is generated by $\scO(D_i)$
for $i = 0, \ldots, n-1$,
and the map
\begin{align*}
\begin{array}{cccc}
 \deg : & \Pic X_\Sigma & \to & \bZ^n \\
  & \vin & & \vin \\
  & \scL & \mapsto & (\deg \scL|_{D_i})_{i=1}^n
\end{array}
\end{align*}
is an isomorphism
of abelian groups.
Since $\Yv = X_\Sigma \setminus D$ is the complement
of a principal divisor,
the restriction map $\Pic X_\Sigma \to \Pic \Yv$
and hence the map \eqref{eq:PicYv} is an isomorphism
of abelian groups.

\section{Lagrangian submanifolds fibred over paths}
 \label{sec:Lagr_submfd_over_paths}

%In Section \ref{sec:Lagr_submfd_over_paths},
%we construct a class of noncompact Lagrangian submanifolds in $(Y,\omega)$.
%This construction was originally due to Donaldson \cite{D00} and has been used extensively e.g. in the works of Seidel \cite{S00,S01,S08}, Khovanov-Seidel \cite{KS02} and Pascaleff \cite{Pascaleff_FCMPP}.

Consider the projection map
\begin{equation*}
\pi:Y\to\bC^\times,\quad (u,v,z)\mapsto z.
\end{equation*}
Each fiber is a conic $\{(u,v)\in\bC^2 \mid uv=z^{-1}f(z)\}$ which
degenerates to the cone $uv=0$ over a zero of the polynomial $f$.
Recall that we denote by
$$
\Delta = \{a_0,a_1,\ldots,a_n\}
$$
the set of zeros of the polynomial $f(z)$,
which we have assumed to be real and positive,
and in the ordering
\begin{equation*}
0<a_0<a_1<\ldots<a_n.
\end{equation*}

The symplectic fibration $\pi:Y\to\bC^\times$ induces a connection
on the tangent bundle of $Y$ where the horizontal subspace at $y\in Y_0$
is given by the symplectic orthogonal complement to the vertical subspace:
$$H_y:=\textrm{Ker}(d\pi_y)^{\perp_\omega}.$$
%Alternatively, since $(Y,\omega)$ is K\"ahler with respect to the natural complex structure on $Y$,
%the horizontal subspace can be regarded
%as the orthogonal complement to $\textrm{Ker}(d\pi_y)$ with respect to the K\"ahler metric.

Given a smooth embedded path $\gamma:I\to\bC^\times\setminus\Delta$ where $I\subset\bR$ is an interval, parallel transport with respect to the above connection along $\gamma$ yields symplectomorphisms between smooth fibers
\begin{equation*}
\tau_{t_1}^{t_2}:\pi^{-1}(\gamma(t_1))\to\pi^{-1}(\gamma(t_2)),
\end{equation*}
for $t_1<t_2 \in I$. These symplectomorphisms are $S^1$-equivariant since the $S^1$-action is Hamiltonian. The vanishing cycle $V_t$ in a smooth fiber $\pi^{-1}(\gamma(t))$ is given by its equator $\mu=0$ (i.e. $|u|=|v|$). Vanishing cycles are invariant under parallel transport, so that the \textit{Lefschetz thimble} $L_\gamma$ along $\gamma$, given by the union of all vanishing cycles $V_t\subset\pi^{-1}(\gamma(t))$ for $t\in I$, is a Lagrangian submanifold in $(Y,\omega)$. If $\gamma$ is a matching path, i.e. a path connecting two critical values of $\pi$ in $\bC^\times$, then $L_\gamma$ is a Lagrangian sphere in $(Y,\omega)$.
%This is how Lagrangian spheres are constructed in \cite{D00,S00,S01,KS02} (see also \cite{Ishii-Ueda-Uehara_SCAS, Chan_HMSARTD}).

Other Lagrangian submanifolds can be constructed in a similar way,
as in the work of Pascaleff \cite{Pascaleff_FCMPP}.
%We proceed with the following definition (recall that $s_i=\log a_i$ for $0\leq i\leq n$):
\begin{defn}\label{admissible}
Let $\gamma:\bR\to\bC^\times\setminus\Delta$ be a smooth embedded path.
We call the path $\gamma$ {\em admissible} if the following conditions are satisfied:
\begin{enumerate}
\item[(1)]
(boundary conditions) $\lim_{t \to -\infty} |\gamma(t)| = 0$ and $\lim_{t \to \infty} |\gamma(t)| = \infty$.
\item[(2)]
$\gamma$ intersects transversally with each of the line segments $\epsilon_i := [a_{i-1}, a_i]$ for $i = 1,\ldots,n$.
\end{enumerate}
\end{defn}

Given an admissible path $\gamma:\bR\to\bC^\times$, we fix $t_0 \in \bR$
and choose a Lagrangian cycle $C_0$ in the conic fiber $\pi^{-1}(\gamma(t_0))$.
Then the submanifold $L_{\gamma,C_0}$ contained inside $\pi^{-1}(\gamma)$ swept out
by the parallel transport of $C_0$ along $\gamma$ is a Lagrangian submanifold in $(Y,\omega)$
(cf. \cite[Section 5.1]{Auroux_MSTD} and \cite[Section 3.3]{Pascaleff_FCMPP}).

\begin{defn} \label{df:winding_number}
Let $\gamma:\bR\to\bC^\times\setminus\Delta$ be an admissible path.
For $i=1,\ldots,n$, we define the {\em $i$-th winding number} $w_i(\gamma)$ of $\gamma$
to be the topological intersection number between $\gamma(\bR)$
and the line segment $\epsilon_i = [a_{i-1}, a_i] \subset \bCx$.
\end{defn}

Notice that the winding numbers of $\gamma$ are invariant
when we deform $\gamma$ in a fixed homotopy class
relative to the boundary conditions
$\lim_{t \to -\infty} |\gamma(t)| = 0$ and
$\lim_{t \to \infty} |\gamma(t)| = \infty$.
In particular, we can deform $\gamma$ so that
$\gamma(t)$ lies on the negative real axis for $t < -T$
for some fixed positive $T$. Then we consider
the Lagrangian (real locus)
\begin{equation*}
C_t:=\{(u,v,z)\in\pi^{-1}(\gamma(t)) \mid u,v \in \bR\}.
\end{equation*}
in the conic fiber $\pi^{-1}(\gamma(t))$ for each $t<-T$,
%(see Figure \ref{fg:fiber_wrapping0}),
which is `dual' to the vanishing cycle $V_t$.
Notice that $C_t$ is invariant under symplectic parallel transport
for all $t<-T$, meaning that
\begin{equation*}
\tau_{t_1}^{t_2}(C_{t_1})=C_{t_2}
\end{equation*}
for all $t_1<t_2<-T$. Let $L_\gamma$ be the submanifold in $Y$
swept out by parallel transport of $C_t$ ($t<-T$):
\begin{equation*}
L_\gamma:=\bigcup_{t\in(-\infty,-T)}C_t\cup\bigcup_{t\in[-T,\infty)}\tau_{-T}^t(C_{-T}).
\end{equation*}
This defines a Lagrangian submanifold in $(Y,\omega)$ which is homeomorphic to $\bR^2$.
Note that the curve $C_{-T}:=\{(u,v,z) \in \pi^{-1}(\gamma(-T)) \mid u,v \in \bR\}$
in the conic fiber $\pi^{-1}(\gamma(-T))$ is being twisted by Dehn twists
in vanishing cycles (i.e. the equator $\lambda=0$ in a fiber of $\pi:Y\to\bC^\times$)
as one goes along $\gamma$.

Furthermore, given an admissible path $\gamma$, we can deform it
(again with respect to the boundary conditions
$\lim_{t \to -\infty} |\gamma(t)| = 0$ and
$\lim_{t \to \infty} |\gamma(t)| = \infty$) so that
the following condition is satisfied:

\begin{defn}\label{strongly_admissible}
%Let $\gamma:\bR\to\bC^\times\setminus\Delta$ be an admissible path.
%We call $\gamma$ {\em strongly admissible}
%if $\gamma$ intersects each circle in $\bC^\times$ centered at the origin once,
%i.e. $|\gamma(\bR)\cap\{z\in\bC^\times||z|=e^s\}|=1$ for all $s\in\bR$,
%or equivalently,
An admissible path $\gamma:\bR\to\bC^\times\setminus\Delta$ is {\em strongly admissible}
if the modulus $|\gamma| : \bR \to \bR^{> 0}$ is strictly increasing.
\end{defn}

%The strong admissibility of $\gamma$ ensures
%that the noncompact Lagrangian submanifold $L_\gamma$ is a section
%of the SYZ fibration $\rho:Y\to B$:
\begin{prop}\label{prop3.1}
Let $\gamma:\bR\to\bC^\times$ be a strongly admissible path.
Then the Lagrangian submanifold $L_\gamma$ is
a section of the SYZ fibration $\rho:Y\to B$.
\end{prop}
\begin{proof}
By definition, $L_\gamma$ is invariant under parallel transport,
so it is a Lagrangian submanifold in $(Y,\omega)$. Moreover,
it is clear that the moment map $\mu$ when restricted to $C_t$ for $t<-T$
is a one-to-one map. Since the symplectomorphisms $\tau_{-T}^t$
are all $S^1$-equivariant, the value of $\mu$ is preserved under parallel transport.
Thus the restriction of $\mu$ to $\tau_{-T}^t(C_{-T})$ remains a one-to-one map.
Together with the condition that $|\gamma(t)|$ is strictly increasing,
this implies that $L_\gamma$ intersects each fiber $T_{s,\lambda}$
of the SYZ fibration $\rho:Y\to B$ at one point.
Hence $L_\gamma$ is a section of $\rho:Y\to B$.
\end{proof}

By deforming $\gamma$ in a fixed homotopy class
%(relative to the ends $\gamma|_{(-\infty,-T]}$ and $\gamma|_{[T,\infty]}$)
%and changing the parametrization
if necessary,
one can furthermore assume that a strongly admissible path
$\gamma : \bR \to \bCx \setminus \Delta$
satisfies $|\gamma(s)| = e^s$ for any $s \in \bR$ and
$\gamma(s_i) = - a_i$ for $i = 0, \ldots, n$ (recall
that $s_i = \log a_i$ for each $i$).
%In this case, the winding numbers can be computed as follows:
Then there exists a unique lift
$
 \gammatilde : \bR \to \bR
$
of
$
 \gamma / |\gamma| : \bR \to S^1
$
such that $\gammatilde(s_0)=0$,
and the winding numbers of $\gamma$ can be computed as
\begin{equation*}
 w_i(\gamma) = \frac{1}{2\pi}\left(\tilde{\gamma}(s_i)-\tilde{\gamma}(s_{i-1})\right)
\end{equation*}
for $i = 1, \ldots, n$.

As an example, consider the path
$$
 \gamma_0:\bR\to\bC^\times,\quad t\mapsto-e^t,
$$
which is clearly strongly admissible.
The corresponding Lagrangian submanifold
$L_0:=L_{\gamma_0}$ is simply the Cartesian product
of the real locus $C_{s_0}$ inside the conic fiber $C_{s_0}=\pi^{-1}(-a_0)$
with the negative real axis (which is the image $\gamma_0(\bR)$ of the path).
As a section of the SYZ fibration $\rho:Y\to B$,
the Lagrangian $L_0$ can be explicitly written as
\begin{eqnarray*}
\sigma:B\to Y,\quad (s,\lambda)\mapsto(u(s,\lambda),v(s,\lambda),z(s,\lambda)),
\end{eqnarray*}
where
\begin{eqnarray*}
u(s,\lambda) & = & \sqrt{\sqrt{\lambda^2+|f(-e^s)|^2}+\lambda},\\
v(s,\lambda) & = & (-1)^{n-1}\sqrt{\sqrt{\lambda^2+|f(-e^s)|^2}-\lambda},\\
z(s,\lambda) & = & -e^s.
\end{eqnarray*}

\section{SYZ transformations}
 \label{sec:SYZ_transform}

%In this section, we will show how SYZ transformations of the noncompact Lagrangian submanifolds in $(Y,\omega)$ that we introduce in the previous section give rise to holomorphic line bundles over the mirror $\check{Y}$.

Consider $\Lambda^\vee\subset T^*B^\sm$, the family of lattices locally generated by $dx_1,dx_2$. As before, $x_1,x_2$ here denote local affine coordinates on $B^\sm$. As in \cite{Chan_HMSARTD}, we will assume that $x_2=-\lambda$ is the globally defined coordinate. Let
\begin{equation*}
\omega_0:=dx_1\wedge d\xi_1+dx_2\wedge d\xi_2
\end{equation*}
be the canonical symplectic structure on the quotient $T^*B^\sm/\Lambda^\vee$ of the cotangent bundle $T^*B^\sm$ by $\Lambda^\vee$, where $(\xi_1,\xi_2)$ are fiber coordinates on $T^*B^\sm$ so that $(x_1,x_2,\xi_1,\xi_2)\in T^*B^\sm$ denotes the cotangent vector $\xi_1dx_1+\xi_2dx_2$ at the point $(x_1,x_2)\in B^\sm$. In the previous section, we have constructed a global Lagrangian section $L_0$ of the SYZ fibration $\rho:Y\to B$. Then a theorem of Duistermaat \cite{Duistermaat_GAAC} says that there exists a fiber-preserving symplectomorphism
\begin{equation*}
\Theta:(T^*B^\sm/\Lambda^\vee,\omega_0)\overset{\cong}{\longrightarrow}(\rho^{-1}(B^\sm),\omega)
\end{equation*}
so that $L_0$ corresponds to the zero section of $T^*B^\sm/\Lambda^\vee$.

This symplectomorphism can be constructed as follows. Let $b\in B^\sm$. Then for every $\alpha\in T^*_bB^\sm$, we can associate a vector field $v_\alpha$ on the fiber $\rho^{-1}(b)$ by
\begin{equation*}
\iota_{v_\alpha}\omega=\rho^*\alpha.
\end{equation*}
Let $\phi^\tau_\alpha$ be the flow of $v_\alpha$ at time $\tau\in\bR$. Then we define an action $\theta_\alpha$ of $\alpha$ on $\rho^{-1}(b)$ by the time-1 flow
\begin{equation*}
\theta_\alpha(y)=\phi^1_\alpha(y),\quad y\in\rho^{-1}(b).
\end{equation*}
Now the covering map
\begin{equation*}
T^*B^\sm\to\rho^{-1}(B^\sm),\quad \alpha\mapsto\theta_\alpha(\sigma(\textrm{pr}(\alpha))),
\end{equation*}
where $\textrm{pr}:T^*B^\sm\to B^\sm$ denotes the projection map, induces a symplectomorphism $\Theta:T^*B^\sm/\Lambda^\vee\to\rho^{-1}(B^\sm)$.

Now let $y_1,y_2$ be the fiber coordinates on $TB^\sm$ dual to $\xi_1,\xi_2$, i.e. $(x_1,x_2,y_1,y_2)\in TB^\sm$ denotes the tangent vector $y_1\partial/\partial x_1+y_2\partial/\partial x_2$ at the point $(x_1,x_2)\in B^\sm$. Given a strongly admissible path $\gamma:\bR\to\bC^\times$, the noncompact Lagrangian submanifold $L_\gamma$ is a section of the SYZ fibration $\rho:Y\to B$ by Proposition \ref{prop3.1} (and as we mentioned before, every admissible path can be deformed to a strongly admissible one). Via the symplectomorphism $\Theta$, we can write $L_\gamma$ in the form
\begin{equation*}
L_\gamma = \left\{(x_1,x_2,\xi_1,\xi_2)\in T^*B^\sm/\Lambda^\vee \mid (x_1,x_2) \in B^\sm,\ \xi_j = \xi_j(x_1,x_2)\textrm{ for $j=1,2$}\right\},
\end{equation*}
where $\xi_j=\xi_j(x_1,x_2)$, $j=1,2$, are smooth functions on $B^\sm$. Since $L_\gamma$ is Lagrangian, the functions $\xi_1,\xi_2$ satisfy
\begin{equation*}
\frac{\partial\xi_j}{\partial x_l}=\frac{\partial\xi_l}{\partial x_j}
\end{equation*}
for $j,l=1,2$ (see \cite{Leung-Yau-Zaslow_SLHYM,Arinkin-Polishchuk_FCFT,Chan_HMSARTD}).

Lying at the basis of the SYZ proposal \cite{Strominger-Yau-Zaslow}
is the fact that the dual $T^*$ of a torus $T$ can be viewed as
the moduli space of flat $U(1)$-connections on the trivial line bundle
$\underline{\bC}:=\bC\times T$ over $T$. So a Lagrangian section
$L=\left\{(x,\xi(x))\in T^*U/T^*U\cap\Lambda^\vee \mid x\in U \right\}$ over an open set $U\subset B^\sm$ corresponds to a family of connections $\{\nabla_{\xi(x)} \mid x\in U\}$ patching together to give a $U(1)$-connection which can locally be written as
\begin{equation*}
\check{\nabla}_U=d+2\pi\sqrt{-1}\left(\xi_1 dy_1 + \xi_2 dy_2\right)
\end{equation*}
over $TU/TU\cap\Lambda\subset\check{Y}$. As shown in \cite{Leung-Yau-Zaslow_SLHYM, Arinkin-Polishchuk_FCFT} (see also \cite[Section 2]{Chan_HMSARTD}), the $(0,2)$-part $F^{(0,2)}$ (and also $(2,0)$-part $F^{(2,0)}$) of the curvature two form of $\check{\nabla}_U$ is trivial.

Recall that a covering of $B^\sm$ is given by the open sets
\begin{equation*}
U_i:=B_i\setminus[s_i,s_{i+1})\times\{0\},\quad V_i:=B_i\setminus(s_{i-1},s_{i}]\times\{0\}.
\end{equation*}
for $i=0,1,\ldots,n$. Now, the restriction of the Lagrangian section $L_\gamma$ to each $U_i$ (resp. $V_i$) is transformed to a $U(1)$-connection $\check{\nabla}_{U_i}$ over $TU_i/TU_i\cap\Lambda$ (resp. $\check{\nabla}_{V_i}$ over $TV_i/TV_i\cap\Lambda$). These connections can be glued together according to the gluing formulas \eqref{wc+}, \eqref{wc-}. Since the $(0,2)$-part $F^{(0,2)}$ of the curvature two form vanishes, this defines a holomorphic line bundle $\mathcal{L}_\gamma$ over $\check{Y}$.
\begin{defn}
Let $\gamma:\bR\to\bC^\times$ be a strongly admissible path and $L_\gamma$ be the noncompact Lagrangian submanifold in $(Y,\omega)$ associated to $\gamma$. We define the SYZ transformation of $L_\gamma$ to be the holomorphic line bundle $\mathcal{L}_\gamma$ over $\check{Y}$, i.e.
\begin{equation*}
\mathcal{F}(L_\gamma):=\mathcal{L}_\gamma.
\end{equation*}
\end{defn}
Notice that the isomorphism class of $\mathcal{L}_\gamma$ is unchanged when we deform $L_\gamma$ in a fixed Hamiltonian isotopy class (or deforming $\gamma$ is a fixed homotopy class relative to the boundary conditions $\lim_{t \to -\infty} |\gamma(t)| = 0$ and
$\lim_{t \to \infty} |\gamma(t)| = \infty$). Henceforth, we will regard this as defining the SYZ transformation of the Hamiltonian isotopy class of the Lagrangian submanifold $L_\gamma\subset Y$ as an isomorphism class of holomorphic line bundle over $\check{Y}$.

As an immediate example, the SYZ transformation of the zero section $L_0=L_{\gamma_0}$ gives the structure sheaf $\mathcal{O}_{\check{Y}}$ over $\check{Y}$.

The main goal of this section is to compute (the isomorphism class of) the line bundle $\mathcal{L}_\gamma$ in terms of the winding numbers of the path $\gamma$. To begin with, note that the isomorphism class of a line bundle over $\check{Y}$ is determined by the degrees of its restrictions to the exceptional divisors $E_i$ for $i=1,\ldots,n$. Given integers $d_1,\ldots,d_n\in\bZ$, let us denote by $\mathcal{L}_{d_1,\ldots,d_n}$ the line bundle on $\check{Y}$ such that $\deg\mathcal{L}_{d_1,\ldots,d_n}|_{E_i}=d_i$.

Now, given a Lagrangian section
\begin{equation*}
L_\gamma=\{(x_1,x_2,\xi_1,\xi_2)\in T^*B^\sm/\Lambda^\vee \mid (x_1,x_2)\in B^\sm,\ \xi_j=\xi_j(x_1,x_2)\textrm{ for $j=1,2$}\},
\end{equation*}
of the SYZ fibration $\rho:Y\to B$, its SYZ transformation is the connection $\check{\nabla}$ which can locally be expressed as
\begin{equation*}
\check{\nabla}_U=d+2\pi\sqrt{-1}(\xi_1dy_1+\xi_2dy_2).
\end{equation*}
Let $\mathcal{L}$ be the line bundle determined by $\check{\nabla}$. Then the degree of its restriction to $E_i$ is given by
\begin{eqnarray*}
\deg\mathcal{L}|_{E_i} & = & \int_{E_i}\frac{\sqrt{-1}}{2\pi}F_{\check{\nabla}}\\
& = & -\int_{E_i}(d\xi_1\wedge dy_1)\\
& = & -(\xi_1(s_i)-\xi_1(s_{i-1})).
\end{eqnarray*}
We have the second equality because $y_2$ is constant (and $x_2=-\lambda=0$) on $E_i$. Hence the isomorphism class of the line bundle $\mathcal{L}$ is completely determined by the increment of the angle coordinate $\xi_1$ on the Lagrangian section $L$ as one moves from $(s_{i-1},0)$ to $(s_i,0)$.
%If $L=L_\gamma$ is the Lagrangian section associated to a strongly admissible path $\gamma:\bR\to\bC^\times$, then these are precisely the winding numbers of $\gamma$:
%\begin{thm}\label{th:SYZ_transform}
%Let $\gamma:\bR\to\bC^\times$ be a strongly admissible path and $L_\gamma$ be the noncompact Lagrangian submanifold in $(Y,\omega)$ associated to $\gamma$. Then the SYZ transformation of the Hamiltonian isotopy class of $L_\gamma$ is given by the isomorphism class of the line bundle $\mathcal{L}_{-w_1(\gamma),\ldots,-w_n(\gamma)}$.
%\end{thm}

Now we prove Theorem \ref{th:main1}:

\begin{proof}[Proof of Theorem \ref{th:main1}]
Recall that the Hamiltonian isotopy class of the Lagrangian submanifold $L_\gamma$
remains unchanged when we deform $\gamma$ in a fixed homotopy class
relative to the boundary conditions $\lim_{t \to -\infty} |\gamma(t)| = 0$ and
$\lim_{t \to \infty} |\gamma(t)| = \infty$.
So, up to a re-parametrization of $\gamma$, we can deform
the restriction of $\gamma$ to $(s_{i-1},s_i)$ to a path arbitrary close to
the concatenation of $\gamma_0|_{(s_{i-1},s_i)}$ (the negative axis) with
a loop winding around the circle $l_s:=\{z\in\bC^\times||z|=e^{s_i}\}$
for $w_i(\gamma)$ times.
Along $\gamma_0$, the angle coordinate $\xi_1$ is constantly zero, and
$\xi$ increases by one when we wind around $l_s$ once in the anti-clockwise direction.
Hence, the increment $\xi_1(s_i)-\xi_1(s_{i-1})$
is precisely given by the $i$-th winding number $w_i(\gamma)$.
\end{proof}

One way to visualize the Lagrangian submanifold $L_\gamma$ is to observe that the curve $\tau_{s_0}^{s_i}(C_{s_0})$ is given by twisting $C_{s_0}$ in the vanishing cycle for $\sum_{k=1}^iw_i(\gamma)$ times. Correspondingly, the increment of the angle coordinate $\xi_1$ as one goes from $(s_{i-1},0)$ to $(s_i,0)$ is given by the $i$-th winding number $w_i(\gamma)$ of $\gamma$.

\section{Homological mirror symmetry}
 \label{sc:hms}

%In Section \ref{sc:conjectures},
%we show that Thereom \ref{th:main2} follows
%from a series of conjectures.

Let $\scX$ be a smooth toric Fano stack of dimension $d$ and
$\scK$ be the total space of its canonical bundle.
The mirror of $\scX$ is given by a Laurent polynomial
$$
 W : (\bCx)^d \to \bC
$$
whose Newton polytope coincides with the fan polytope of $\scX$.
By choosing sufficiently general $W$,
one may assume that
\begin{itemize}
 \item
$W$ is {\em tame}
in the sense that the gradient $\| \nabla W \|$ is bounded from below
by a positive number
outside of a compact set,
 \item
all the critical points of $W$ are {\em non-degenerate}
in the sense that the Hessian at each critical point
is a non-degenerate quadratic form,
 \item
all the critical values of $W$ are distinct, and
 \item
the origin is not a critical value of $W$.
\end{itemize}
When $\cX = \bP(1, n)$ is the weighted projective line of weight $(1,n)$,
%of weight $(1, n)$,
one can take
$$
 W : \bCx \to \bC, \qquad z \mapsto z + 1 / {z^n}
$$
as its mirror.
This is related to the function
$1/z + z^n$ appearing in \eqref{eq:Y'1}
by an automorphism $z \mapsto 1/z$ of the torus,
and our choice here is made
only for aesthetic reason
(Figure \ref{fg:thimbles3} below looks nicer for this choice).
The critical points of $W$ are given by
$
 z = \sqrt[n+1]{n} \zeta_{n+1}^i,
$
$
 i = 0, \ldots, n,
$
with critical values
$
 \sqrt[n+1]{n} (1 + 1/n) \zeta_{n+1}^{i}.
$
Here
$
 \zeta_{n+1} = \exp(2 \pi \sqrt{-1} / (n+1))
$
is a primitive $(n+1)$-st root of unity.

We equip $(\bCx)^d$
with the K\"{a}hler form
\begin{equation*}
 \omega = - \frac{\sqrt{-1}}{2}
  \lb
   \frac{dz_1 \wedge d\bar{z_1}}{|z_1|^2}
    + \cdots
    + \frac{dz_d \wedge d\bar{z_d}}{|z_d|^2}
  \rb,
\end{equation*}
and
define a horizontal distribution on $(\bCx)^d$
as the orthogonal complement
to the tangent spaces to fibers of $W$.
Choose a sufficiently large closed disk $B \subset \bC$
and a point $*$ on the boundary of $B$,
so that all the critical values of $W$
is contained in the interior of $B$.
The restriction $W|_S : S \to B$ of $W$
to the intersection $S$
of $W^{-1}(B)$ and a sufficiently large closed ball in $(\bCx)^d$
%is an {\em exact Lefschetz fibration}
%%\cite[Section 15]{Seidel_PL}
%%\cite[Definition 2.12]{MR2497314},
%\cite{Seidel_PL,MR2497314}
%in the sense that
%it is a family of exact symplectic manifolds
%with non-degenerate critical points.
%In addition, it
is a {\em compact convex Lefschetz fibration}
\cite[Definition 2.14]{MR2497314},
i.e.,
a family of compact convex symplectic manifolds
%with contact type boundary
with at worst non-degenerate critical points.
%the smooth fiber $E_* := \pi^{-1}(*)$ is
%a {\em compact convex symplectic manifold}
%(also known as an {\em exact symplectic manifold
%with contact type boundary},
%or a {\em Liouville domain}).
Here, a {\em compact convex symplectic manifold}
(also known as a {\em Liouville domain})
is an exact symplectic manifold with boundary
whose Liouville vector field
%$Z$, defined by $\iota_Z (d \theta) = \theta$,
points strictly outward along the boundary.
We will write $W|_S$ as $W$ by abuse of notation.
By rounding the corners of $S$,
one obtains a Liouville domain $M$.
Its {\em completion}
$
 \hat{M} = M \cup_{\partial M} (\partial M \times [1,\infty)),
$
obtained by gluing the positive half of the symplectization
to the boundary,
is symplectomorphic to $(\bCx)^d$.

When $W = z + 1/z^n$,
we fix a large positive real number $*$ as a base point,
and let $B$ to be a closed disk of radius $*$
centered at the origin.
Its inverse image $S := W^{-1}(B)$ is symplectomorphic
to a cylinder $[0, L] \times S^1$ for some $L$,
equipped with the standard symplectic form
$\omega = d r \wedge d \theta$.
The fiber $S_* := W^{-1}(*)$ consists of $n+1$ points,
one of which is approximately $*$
and $n$ of which are approximately $n$-th roots of $1/*$.

%There are several ways to associate a Fukaya category
%to a Lefschetz fibration.
%The {\em Fukaya-Seidel category} $\scF(\pi)$
%is defined by Seidel \cite[Definition 18.12]{Seidel_PL}
%%the {\em Fukaya category of Lefschetz fibration}
%%(a.k.a. the Fukaya-Seidel category)
%%as an $A_\infty$-category $\Fuk W$
%as the $\bZ / 2 \bZ$-invariant part
%of the Fukaya category of a branched double cover of $E$.
%of the Lefschetz fibration
%branched along a smooth fiber.
%of $W$
%The Lefschetz thimbles $(\Delta_i)_{i=1}^m$
%give rise to Lagrangian spheres
%$(\Deltatilde_i)_{i=1}^m$ in $\Etilde$
%branched along $W^{-1}(\ast)$
%which, together with auxiliary Floer-theoretic data,
%yields a full exceptional collection in $D^b \Fuk W$
%by \cite[Propositions 18.14 and 18.17]{Seidel_PL}.
With a compact convex Lefschetz fibration,
Abouzaid \cite{Abouzaid_HCRMSTV,Abouzaid_MHTGHMSTV}
associates a category $\scF(S, S_*)$
consisting of exact Lagrangian submanifolds $L$ with boundaries,
which are {\em admissible} in the sense that
\begin{itemize}
 \item
the boundary $\partial L$ is a Lagrangian submanifold
in the interior of $S_*$,
%over the base point $* \in \partial S$,
 \item
$L$ projects by $W$ to a curve $\gamma \subset B$
in a neighborhood of $\partial L$, and
 \item
$L$ coincides with the parallel transport of $\partial L$
along $\gamma$ in this neighborhood.
\end{itemize}
For a pair $(L_1, L_2)$ of admissible Lagrangian submanifolds,
%$L_i$ for $i = 1, 2$
%associated with $z_i \in S$ and $L_{z_i} \subset S_{z_i}$,
the space $\hom_{\cF(S, S_*)}(L_1, L_2)$ of morphisms
is defined as the Lagrangian intersection Floer complex
$CF(L_1^{\epsilon_1}, L_2)$
between $L_1^{\epsilon_1}$ and $L_2$.
Here $L_1^{\epsilon_1}$ is the Lagrangian submanifold
obtained by perturbing the part of $L_1$
fibered over $\gamma_1$
to another Lagrangian submanifold
fibered over the curve $\gamma_1^{\epsilon_1}$ starting from $*$
such that
$
 \arg {\gamma_1^{\epsilon_1}}'(0) < \arg \gamma_2'(0) < \arg {\gamma_1^{\epsilon_1}}'(0) + \pi
$
as in Figure \ref{fg:adlag2},
and one only looks at intersection points
in the interior of $S$.
For a sequence $(L_1, \ldots, L_k)$ of admissible Lagrangians,
the $A_\infty$-operation is defined by perturbing $L_i$
to $L_i^{\epsilon_i}$ such that
$
 \arg {\gamma_1^{\epsilon_1}}'(0) < \arg {\gamma_2^{\epsilon_2}}'(0)
  < \cdots < \arg {\gamma_k^{\epsilon_k}}'(0) < \arg {\gamma_1^{\epsilon_1}}'(0) + \pi
$
and counting virtual numbers of holomorphic disks
bounded by these Lagrangian submanifolds.
\begin{figure}[htbp]
\begin{minipage}{.5 \linewidth}
\centering
\input{adlag1.pst}
\caption{Admissible Lagrangians}
\label{fg:adlag1}
\end{minipage}
\begin{minipage}{.5 \linewidth}
\centering
\input{adlag2.pst}
\caption{Perturbing $L_1$}
\label{fg:adlag2}
\end{minipage}
\end{figure}

%Examples of admissible Lagrangian submanifolds
%is given by a Lefschetz thimble.
A {\em vanishing path} is an embedded path
$\gamma : [0,1] \to B$
from the base point $*$ to a critical value of $\pi$
avoiding other critical values.
%For later convenience,
We assume that $\gamma$ does not pass through the origin.
By arranging the vanishing cycles along a vanishing path,
one obtains an admissible Lagrangian submanifold of $S$
called the {\em Lefschetz thimble}.
A {\em distinguished basis of vanishing paths}
is a sequence $\bsgamma = (\gamma_1, \ldots, \gamma_m)$
of mutually non-intersecting vanishing paths,
one for each critical value
and ordered according to $- \arg \gamma_i'(0)$.
The corresponding sequence of Lefschetz thimbles will be written as
$
 \bsDelta = (\Delta_1, \ldots, \Delta_m),
$
and the full subcategory of $\scF(S, S_*)$
consisting of $\bsDelta$ will be denoted by
$\scF(\bsDelta)$.
It is expected that
\begin{itemize}
 \item
$\bsDelta$ is a full exceptional collection in $D^b \scF(S, S_*)$,
so that $D^b \scF(\bsDelta)$ is equivalent to $D^b \scF(S, S_*)$
and hence is independent of the choice of $\bsDelta$,
and
 \item
$D^b \scF(S, S_*)$ is equivalent
to the derived category $D^b \scF(W)$
of the {\em Fukaya-Seidel category} of $W$,
defined in \cite[Definition 18.12]{Seidel_PL}
as the $\bZ / 2 \bZ$-invariant part
of the Fukaya category of a branched double cover of
a slight enlargement of $S$.
\end{itemize}

Homological mirror symmetry for toric Fano stacks
can be formulated as follows:

\begin{conjecture}[{Kontsevich \cite{Kontsevich_ENS98}}]
 \label{cj:hms_toric_Fano}
There exists an equivalence
\begin{align} \label{eq:hms_toric_Fano}
 \Psi : D^b \scF(\bsDelta) \to D^b \coh \scX
\end{align}
of triangulated categories.
%sending $\Deltatilde_i$ to $E_i \in D^b \coh \scX$.
\end{conjecture}

Let $\sigma : S \to S$ be a Hamiltonian diffeomorphism
which covers the Dehn twist
$\sigmabar : B \to B$
along a circle near the boundary of $B$.
%\cite[Proposition 18.23]{Seidel_PL}.
The resulting push-forward functor
$\sigma_* : \scF(S, S_*) \to \scF(S, S_*)$
is an autoequivalence of the Fukaya category,
which wraps a Lagrangian
as shown in Figure \ref{fg:adlag3}.
\begin{figure}[htbp]
\begin{minipage}{\linewidth}
\centering
\input{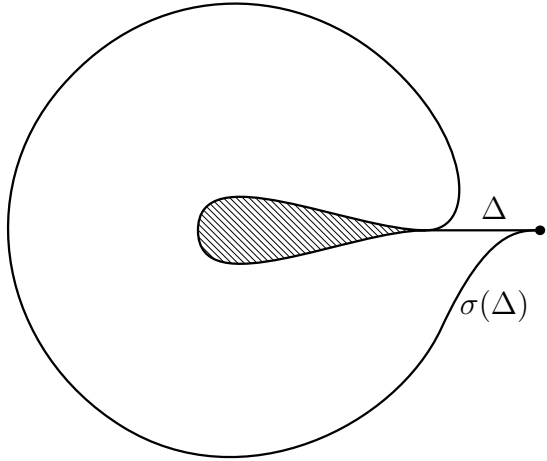}
\caption{The functor $\sigma_*$}
\label{fg:adlag3}
\end{minipage}
\end{figure}

One can equip the direct sum
$$
 \cA = \bigoplus_{i, j=0}^m \bigoplus_{k=0}^\infty
  \Hom_{\cF(S, S_*)}(\sigma^{k}(\Delta_i), \Delta_j)
$$
with a ring structure given by
$$
\begin{array}{ccccc}
 \Hom(\sigma^{k_2}(\Delta_{i_2}), \Delta_{i_3})
  & \otimes &
 \Hom(\sigma^{k_1}(\Delta_{i_1}), \Delta_{i_2})
  & \to &
 \Hom(\sigma^{k_1+k_2}(\Delta_{i_1}), \Delta_{i_3}) \\
 \vin & & \vin & & \vin \\
 x_2 & & x_1 & \mapsto &\frakm_2(x_2, \sigma_*^{k_2}(x_1)).
\end{array}
$$
The continuation map
$$
 \frakt : \Hom_{\cF(S, S_*)}(\sigma^k(\Delta_i), \Delta_j)
  \to \Hom_{\cF(S, S_*)}(\sigma^{k+1}(\Delta_i), \Delta_j)
$$
in Floer theory,
obtained by counting solutions to inhomogeneous Cauchy-Riemann equation
(cf. e.g \cite[(3.35)]{Abouzaid-Seidel_OSAVF}),
induces an endomorphism $\frakT$ of this ring.

On the mirror side,
there is a distinguished autoequivalence
\begin{align*}
 S(-) = - \otimes \omega_\scX[d] : D^b \coh \scX \to D^b \coh \scX
\end{align*}
called the {\em Serre functor}
\cite{Bondal-Kapranov_Serre}.
Let $s \in H^0(\omega_\scX^\vee)$ be a section
characterized by the property that
the zero locus $s^{-1}(0)$ is the union of all the toric divisors of $\scX$.
This section induces a natural transformation
\begin{align} \label{eq:nat_transf2}
 \fraks : S[-d] \to \id
\end{align}
which acts on objects by multiplication by $s$ :
$
 S[-d](X) = X \otimes \omega_\scX \to \id(X) = X.
$
One can equip the direct sum
$$
 \cB =
 \bigoplus_{i, j=0}^m \bigoplus_{k=0}^\infty
  \Hom_\scX(E_i \otimes \omega_\scX^k, E_j)
$$
with a ring structure given by
\begin{align*}
 \Hom(E_{i_2} \otimes \omega_\cX^{k_2}, E_{i_3})
  &\otimes
 \Hom(E_{i_1} \otimes \omega_\cX^{k_1}, E_{i_2}) \\
  &\simto
 \Hom(E_{i_2} \otimes \omega_\cX^{k_2}, E_{i_3})
  \otimes
 \Hom(E_{i_1} \otimes \omega_\cX^{k_1+k_2}, E_{i_2} \otimes \omega_\cX^{k_2}) \\
  &\to
 \Hom(E_{i_1} \otimes \omega_\cX^{k_1+k_2}, E_{i_3}).
\end{align*}
The natural transformation \eqref{eq:nat_transf2}
%%induced by the natural transformation
%%consisting of
%$$
% \Hom(E_i \otimes \omega_\scX^{k}, E_j)
%  \to \Hom(E_i \otimes \omega_\scX^{k+1}, E_j)
%$$
%defined by multiplication by $x y \in H^0(\omega_\scX^\vee)$
induces a ring endomorphism $\frakS$ of $\cB$.

Assume that the collection
%$(\Delta_i)_{i=1}^m$ and
$(E_i)_{i=1}^m$ is {\em cyclic}
in the sense that
$$
 \Ext^i(E_k \otimes \omega_\cX^j, E_\ell) = 0
  \ \text{for any $i \ne 0$, any $j \ge 0$ and any $k, \ell \in \{ 1, \ldots, m\}$}.
$$
This implies that the ring $\cB$ is concentrated in degree zero,
so that there are no higher $A_\infty$-operations
for degree reason.
Although cyclicity is a strong condition,
it is known that any toric Fano stack in dimension two
has a cyclic full exceptional collection of line bundles
\cite{Ishii-Ueda_DMEC}.
It is not known whether any toric Fano stack
has a cyclic full exceptional collection of complexes.

\begin{conjecture}[{Kontsevich, Seidel \cite[Example 5]{Seidel_SHHH}}] \label{cj:ring_isom}
There is a ring isomorphism
$\cA \simto \cB$
sending the ring endomorphism
$\frakT$ to $\frakS$.
\end{conjecture}

It is further expected that $\frakt$ is the first component
of a canonical natural transform
\begin{align} \label{eq:nat_transf1}
 \frakt : \sigma_* \to \id
\end{align}
which is mirror to the natural transformation \eqref{eq:nat_transf2}.

When $W = z + 1/z^n$,
we take a distinguished basis
$(\gamma_i)_{i=0}^n$
of vanishing paths as in Figure \ref{fg:thimbles1}.
The corresponding Lefschetz thimbles $(\Delta_i)_{i=0}^n$
are Lagrangian submanifolds of $S$
with boundaries on $S_*$
as shown in \pref{fg:thimbles3}.
The Hamiltonian diffeomorphism
$\sigma : S \to S$ behaves as
$z \mapsto \exp(2 \pi \sqrt{-1}) z$ for $|z| \gg 1$, and
$z \mapsto \exp(- 2 \pi \sqrt{-1} / n) z$ for $|z| \ll 1$.
We write the images of the Lefschetz thimbles as
$\Delta_{- (n+1) k + i} = \sigma^k(\Delta_i)$.

\begin{figure}[htbp]
\begin{tabular}[b]{cc}
\begin{minipage}{.5 \linewidth}
\centering
\vspace{7mm}
\input{thimbles1.pst}
\caption{Vanishing paths on the $W$-plane}
\label{fg:thimbles1}
\end{minipage}
&
\begin{minipage}{.5 \linewidth}
\centering
\vspace{-10mm}
\input{thimbles3.pst}
\caption{Lefschetz thimbles on the $z$-plane}
\label{fg:thimbles3}
\end{minipage}
\end{tabular}
\end{figure}

%The direct sum
%$
% \bigoplus_{i, j=0}^m \bigoplus_{k=0}^\infty
%  \Hom_\scX(E_i \otimes \omega_\scX^k, E_j)
%$
%has a ring structure given by
%$
% y_2 \cdot y_1
%  = y_2 \circ (S[-n])^{k_2}(y_1)
%$
%where
%$
% y_\ell \in \Hom_\scX(E_{i_\ell} \otimes \omega_\scX^{k_\ell}, E_{j_\ell})
%$
%for $\ell = 1, 2$
%and
%$
% (S[-n])^{k_2}(y_1)
%  \in \Hom_\scX(E_{i_\ell} \otimes \omega_\scX^{k_1+k_2},
%   E_{j_\ell} \otimes \omega_\scX^{k_2}).
%$

On the mirror side,
we write the homogeneous coordinate ring of $\scX = \bP(1, n)$ as $\bC[x,y]$
where $\deg x = 1$ and $\deg y = n$.
The canonical bundle is given by $\omega_\scX = \scO_\cX(-n-1)$,
and the full exceptional collection
$$
 (E_0, E_1, \ldots, E_n) = (\scO_\cX, \scO_\cX(1), \ldots, \scO_\cX(n))
$$
of line bundles is cyclic.

%\pref{th:ring_isom} below gives \pref{cj:hms_toric_Fano} and
%a part of \pref{cj:nat_transf}
%for $\cX = \bP(1, n)$:
%which are needed for the proof of Theorem \ref{th:wrapped_hms}:

\begin{theorem} \label{th:ring_isom}
\pref{cj:ring_isom} holds
for $\scX = \bP(1,n)$ and $W = z + 1/z^n$.
\end{theorem}

\begin{proof}
The Lefschetz thimble $\Delta_{-k(n+1)+i}$ starts
from the critical point $\sqrt[n+1]{n} \zeta_{n+1}^i$
and extends in two directions:
One wraps around the origin $k/n$ times and
goes to the point in $S_*$
approximated by $\zeta_n^{i-k} \cdot *^{-1/n}$.
The other wraps around infinity $k$ times and
asymptotes to the point in $S_*$
approximated by $*$.
Figure \ref{fg:thimbles6} shows the picture of two Lefschetz thimbles.
Here the top and the bottom edges of the rectangle are identified
to form the cylinder $[0, L] \times S^1$,
which is symplectomorphic to $S$.
The vertical dotted line shows the locus
where the absolute value is $\sqrt[n+1]{n}$.
The Lefschetz thimble $\Delta_0$ is just the real line,
and the Lefschetz thimble $\Delta_{-k(n+1)+i}$
is obtained from the Lefschetz thimble $\Delta_i$
by wrapping $k$ times.

\begin{figure}[htbp]
\centering
\input{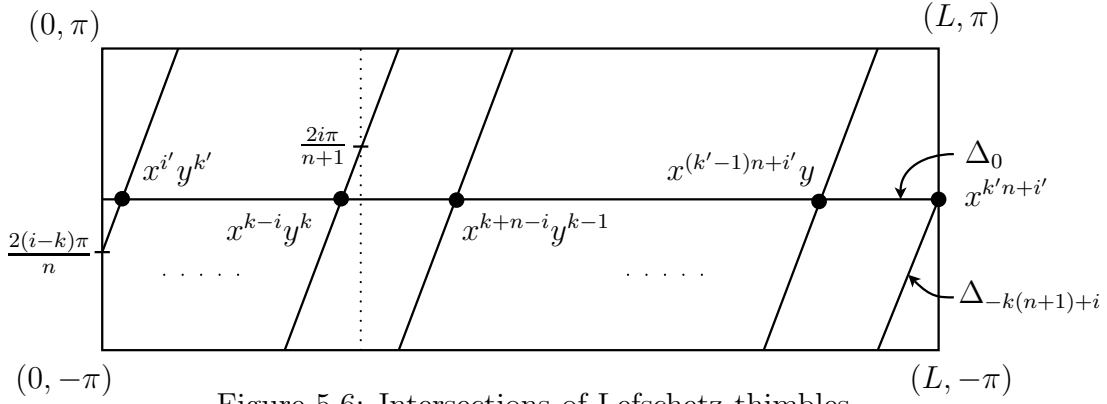}
\caption{Intersections of Lefschetz thimbles}
\label{fg:thimbles6}
\end{figure}

Write $k(n+1) - i = k' n + i'$ for $k' \in \bZ$ and $i' \in \{ 0, \ldots, n-1 \}$.
Then the thimbles $\Delta_{-k(n+1)+i}$ and $\Delta_0$
intersect at $k' + 1$ points,
and we label these intersection points
by the basis of $\Hom(\scO_{\bP(1,n)}(-k(n+1)+i), \scO_{\bP(1,n)})$
as shown in Figure \ref{fg:thimbles6}.
Intersections between other Lefschetz thimbles
can be labeled similarly.
By choosing the standard grading on $S$
determined by the holomorphic volume form $dz/z$
and suitable gradings on $\Delta_i$'s,
one can arrange that $\Hom_{\cF(S, S_*)}(\Delta_i, \Delta_j)$
for $i \le j$ have degree zero
(and those for $i > j$ have degree one).
It follows that the ring
$
 \cA = \bigoplus_{i, j=0}^m \bigoplus_{k=0}^\infty
  \Hom_{\cF(S, S_*)}(\sigma^{k}(\Delta_i), \Delta_j)
$
is concentrated in degree zero,
so that there are no higher $A_\infty$-operations
on $\cA$.
The multiplication on $\cA$
comes from a triangle on $S$,
which is either of the ones
shown in Figures \ref{fg:triangle1} and \ref{fg:triangle2}.
\begin{figure}[htbp]
\begin{minipage}{.5 \linewidth}
\centering
\input{triangle1.pst}
\caption{A triangle}
\label{fg:triangle1}
\end{minipage}
\begin{minipage}{.5 \linewidth}
\centering
\input{triangle2.pst}
\caption{Another triangle}
\label{fg:triangle2}
\end{minipage}
\end{figure}
This clearly matches the multiplication on $\cB$
(which is just the multiplication of polynomials),
and \pref{th:ring_isom} is proved.
\end{proof}

For $q \in \Hom_{\cF(S, S_*)} \lb \sigma^{w}(\Delta_{i_0}), \Delta_{i_1} \rb$,
we define $\ord(q)$ as the maximal integer $d$
such that $x$ is in the image of
$$
 \frakt^d : \Hom_{\cF(S, S_*)} \lb \sigma^{w-d}(\Delta_{i_0}), \Delta_{i_1} \rb
  \to \Hom_{\cF(S, S_*)} \lb \sigma^{w}(\Delta_{i_0}), \Delta_{i_1} \rb.
$$
\pref{pr:degree} below is an analogue of \cite[Proposition 4.3]{Pascaleff_FCMPP},
which will be useful later.

\begin{proposition} \label{pr:degree}
If a holomorphic triangle $\varphi : D^2 \to S$
contributes to the composition
$
 x^{i_2}y^{j_2} = \frakm_2(x^{i_1} y^{j_1}, x^{i_0} y^{j_0})
$
for
$
 x^{i_0} y^{j_0} \in \Hom_{\cF(S, S_*)}(\Delta_{k_0}, \Delta_{k_1}),
$
$
 x^{i_1} y^{j_1} \in \Hom_{\cF(S, S_*)}(\Delta_{k_1}, \Delta_{k_2}),
$
and
$
 x^{i_2} y^{j_2} \in \Hom_{\cF(S, S_*)}(\Delta_{k_0}, \Delta_{k_2}),
$
then the intersection number
between this triangle and the divisor $S_0$ is given by
\begin{align} \label{eq:degree}
 \varphi(D^2) \cdot S_0
  =  \ord(x^{i_2} y^{j_2}) - \ord(x^{i_0} y^{i_0}) - \ord(x^{i_1} y^{j_1}).
\end{align}
\end{proposition}

\begin{proof}
%The proof is parallel to \cite[Proposition 4.3]{Pascaleff_FCMPP}:
Since $\sqrt[n+1]{n}$ is close to $1$,
one can perturb Figure \ref{fg:thimbles6} slightly
to set the dotted vertical line to be the unit circle on $\bCx$
without changing the intersections
of triangles with $S_0$.
Then $S_0$ is equidistributed on the dotted vertical line
with vertical coordinates $(2 j + 1) \pi / (2 n + 2)$ for $j = 0, \ldots, n$.

We write $d_\ell = \ord(x^{i_\ell} y^{j_\ell})$ and
$x^{i_\ell} y^{j_\ell} = (x y)^{d_\ell} x^{i_\ell-d_\ell} y^{j_\ell-d_\ell}$
for $\ell = 0, 1, 2$.
If $i_\ell - d_\ell = 0$ for $\ell = 0, 1$ or
$j_\ell - d_\ell = 0$ for $\ell = 0, 1$,
then the whole triangle is either
on the left or on the right of the dotted vertical line
in Figure \ref{fg:thimbles6},
and both sides of \eqref{eq:degree} is zero.
If $a := i_0 - d_0 > 0$ and $b := j_1 - d_1 > 0$,
then the triangle $D$ looks as in Figure \ref{fg:triangle1},
and the vertical dotted line cuts $D$ into two.
The number of points in $S_0$
on the part of the vertical dotted line
bounded by $\Delta_{i_0}$ and $\Delta_{i_1}$ is given by $a$,
and that by $\Delta_{i_1}$ and $\Delta_{i_2}$ is given by $b$
(here we are working on the universal cover of the cylinder $S$).
The vertical dotted line is on the right or the left
of the vertex $x^{i_2} y^{j_2}$
depending on whether $a > b$ or $a < b$
(and exactly on the line if $a = b$).
It follows that
both sides of \eqref{eq:degree} is given by $\max \{ a, b \}$
in this case.
The case $j_0 - d_0 > 0$ and $i_1 - d_1 > 0$ is similar,
and Proposition \ref{pr:degree} is proved.
\end{proof}

An admissible Lagrangian submanifold $\Delta$ of $S$
can be completed to a Lagrangian submanifold
$\hat{\Delta} := \Delta \cup_{\partial \Delta} (\partial \Delta \times [1, \infty))$
of the completion $\hat{M}$.
Given a pair $(\Delta_1, \Delta_2)$ of admissible Lagrangians,
the {\em wrapped Floer cohomology}
\cite{Abouzaid-Seidel_OSAVF}
%of their completions
is defined by
$$
 \HW(\hat{\Delta}_1, \hat{\Delta}_2) := \varinjlim_w \HF(\phi_w(\hat{\Delta}_1), \hat{\Delta}_2),
$$
where $\phi_w$ is the time $w$ flow
by a Hamiltonian $H$
which behaves as
$
 H(x, r) = r
$
on the cylindrical end
$(x, r) \in \partial M \times [1, \infty)$
of $\hat{M}$.
Note that if $\Delta$ is compact in the fiber direction of $W$,
then $\phi_w(\hat{\Delta})$ is isotopic to the completion of $\sigma^w(\Delta)$
through compactly-supported Hamiltonian diffeomorphisms.
If $\Delta$ is mirror to $\scO_\cX$,
then Conjectures \ref{cj:hms_toric_Fano} and \ref{cj:ring_isom} give
\begin{align*}
 \HW(\hat{\Delta}, \hat{\Delta})
  &\cong \varinjlim_w HF(\phi_w(\hat{\Delta}), \hat{\Delta}) \\
  &\cong \varinjlim_w HF(\sigma^w(\Delta), \Delta) \\
  &\cong \varinjlim_w \Hom(\omega_\cX^{\otimes w}, \cO_\cX) \\
  &\cong \varinjlim_w H(\omega_\cX^{\otimes (-w)}) \\
  &\cong \bC[\scX \setminus s^{-1}(0)],
\end{align*}
where
$
 \bC[\cX \setminus s^{-1}(0)]
  \cong \bC[x_1^{\pm 1}, \ldots, x_d^{\pm 1}]
$
is the coordinate ring of the dense torus of $\cX$
obtained by removing the anti-canonical divisor $s^{-1}(0)$
from $\cX$.

%The {\em suspension} of $W : (\bCx)^n \to \bC$ is defined by
%$$
% W^\sigma(\bsz, u) = W(\bsz) - u^2
%  : (\bCx)^n \times \bC \to \bC.
%$$
%The double suspension is equivalent to
%$$
% W^{\sigma \sigma} = W(\bsz) - u v
%  : (\bCx)^n \times \bC^2 \to \bC
%$$
%after a change of coordinates.
Let
$$
 Y' := \{ (\bsz, u, v) \in (\bCx)^d \times \bC^2 \mid W(\bsz) = u v \}
$$
be the fiber of the double suspension of $W$.

%The fiber of the double suspension
%If $W$ is mirror to $\cX$,
%then $Y'$ is mirror to $\cK$:
%This corresponds to the total space of the canonical bundle
%of the mirror:
%If $D^b \scF(W)$ is equivalent to $D^b \coh \scX$,
%then the double suspension
%is mirror to the total space of the canonical bundle:
%The following theorem is due to Seidel:

\begin{theorem}[{Seidel \cite[Theorem 1]{Seidel_suspension}}]
 \label{th:Seidel_suspension}
Assume that the derived Fukaya-Seidel category $D^b \scF(W)$ is quasi-equivalent
to the derived category of coherent sheaves on $\scX$ as an $A_\infty$-category.
Then there exists a full embedding
$$
 D^b \coh_0 \scK \hookrightarrow D^b \cF(Y')
$$
of triangulated categories,
where $D^b \coh_0 \scK$ is the full subcategory of $D^b \coh \scK$
consisting of complexes whose cohomologies
are supported on the zero-section, and
$\cF(Y')$ is the Fukaya category of $Y'$.
\end{theorem}

The manifold $Y'$ is an affine algebraic variety,
and hence a Stein manifold of finite type,
so that it is symplectomorphic to the completion $\hat{N}$
of the Liouville domain $N$
%$$
% M = \lc (z, u, v) \in Y' \relmid |z|^2 + |z^{-1}|^2 + |u|^2 + |v|^2 \le R' \rc
%$$
obtained by intersecting it with a sufficiently large ball.
Consider the projection
\begin{align*}
\begin{array}{cccc}
 \varpi : & N & \to & (\bCx)^d \\
  & \vin & & \vin \\
  & (\bsz, u, v) & \mapsto & \bsz
\end{array}
\end{align*}
and set
%$
% S = \lc z \in \bCx \mid |z+1/z^n| \le R \rc
%$
%and
$
 E = N \cap \varpi^{-1}(S).
$
The restriction $\varpi|_E : E \to S$ of the projection,
which we will write $\varpi$ by abuse of notation,
is a compact convex Lefschetz fibration,
whose discriminant locus is $S_0 := W^{-1}(0)$.
The completion $\hat{E}$ of $E$,
obtained by first completing in the direction of the fiber of $\varpi$
and then extending it to the completion of the base $S$,
is symplectomorphic to $Y'$.

We define the Lagrangian submanifolds $L_i \subset E$
which are $\bR$-fibrations
over the Lefschetz thimbles
$\Delta_i \subset S$
as in Section \ref{sec:Lagr_submfd_over_paths}.
When $W = z + 1/z^n$,
the Lagrangian submanifold $L_0$ is given by
\begin{align}
 L_0 := \lc (z, u, v) \in E \relmid z, u, v \in \bR^{> 0} \rc,
\end{align}
and other Lagrangian submanifolds
$L_i$ are given by $L_i = \zeta_{n+1}^i \cdot L_0$,
where
%the group
%%$
%% \lc \zeta_{n+1}^i \in \bCx \relmid i = 0, \ldots, n \rc
%%  \cong \bZ / (n+1) \bZ
%%%  \quad \zeta_{n+1} = \exp \lb 2 \pi \sqrt{-1} / (n+1) \rb
%%$
%$
% \lc \zeta_{n+1}^i \rc_{i=0}^n \cong \bZ / (n+1) \bZ
%%  \quad \zeta_{n+1} = \exp \lb 2 \pi \sqrt{-1} / (n+1) \rb
%$
$\zeta_{n+1}^i$
acts on $E$
by sending $(z, u, v)$ to $(\zeta_{n+1}^i z, \zeta_{n+1}^{i} u, v)$.

Recall from \cite[Definition 2.21]{MR2497314}
that a {\em Lefschetz admissible Hamiltonian}
is a smooth function $H(\bsz, u, v)$ on $Y'$
which can be written,
outside of a compact set,
as the sum
$H_b(\bsz) + H_f(u, v)$
of admissible Hamiltonians $H_b(\bsz)$ and $H_f(u, v)$
on the base and the fiber respectively.
The map $\varpi$ is not an honest Lefschetz fibration
but a Bott-Morse analogue of a Lefschetz fibration in general.
It is an honest Lefschetz fibration when
%$\cX= \bP(1,n)$ and
$W = z + 1/z^n$
for dimensional reason.

The {\em Lefschetz wrapped Floer cohomology}
is defined as the direct limit
$$
 \HW_l(\hat{L}_i, \hat{L}_j) := \varinjlim_{w} \HF(\phi_w(\hat{L}_i), \hat{L}_j)
$$
where $\phi_w : Y' \to Y'$ is the time $w$ flow
defined by the Hamiltonian $H$.
Set
\begin{align*}
 \scrA_w := \bigoplus_{i, j=1}^m \HF(\phi_w(\hat{L}_i), \hat{L}_j)
\end{align*}
and
\begin{align*}
 \scrA := \varinjlim_{w \to \infty} \scrA_w
  = \bigoplus_{i, j=1}^m \HW_l(\hat{L}_i, \hat{L}_j).
\end{align*}
The map
$$
 \phi_{w+1, w}^\scrA :
  \HF(\phi_{w}(\hat{L}_i), \hat{L}_j) \to \HF(\phi_{w+1}(\hat{L}_i), \hat{L}_j)
$$
in the definition of the inductive limit
is the continuation map
for the Hamiltonian diffeomorphism $\phi_1 : Y' \to Y'$.
The universal map to the inductive limit is denoted by
$$
 \phi_{\infty,w}^\scrA : \HF(\phi_w(\hat{L}_i), \hat{L}_j) \to \scrA.
$$
The $\bC$-vector space $\scrA$ has a ring structure
coming from
\begin{align*}
\begin{array}{ccc}
 \HF(\phi_{w_2} (\hat{L}_j), \hat{L}_k) \otimes \HF(\phi_{w_1} (\hat{L}_i), \hat{L}_j)
  & \to & \HF(\phi_{w_1+w_2} (\hat{L}_i), \hat{L}_j) \\
 \vin & & \vin \\
 q \otimes r & \mapsto & \frakm_2(q, (\phi_{w_2})_*(r)).
\end{array}
\end{align*}
One can show the ring isomorphism
\cite[Theorem A.2]{Chan-Pomerleano-Ueda_conifold}
\begin{align*}
 \scrA \cong \bigoplus_{i_0,i_1=1}^n \HW(\hat{L}_{i_0}, \hat{L}_{i_1})
\end{align*}
with the ordinary wrapped Floer cohomology
along the lines of \cite[Theorem 2.22]{MR2497314}.
%The `Bott-Morse' Lefschetz case is similar to,
%but easier than,
%the `fibered' case
%\cite[Theorem A.3]{Chan-Pomerleano-Ueda_conifold},
%%in that $Y' = (\bCx)^r \times_\bC \bC^2$ is the fiber product of
%%$W : (\bCx)^r \to \bC$ and $u v : \bC^2 \to \bC$,
%%but it is in fact easier
%in that the Lagrangians $\hat{L}_i$ are compact
%in the $(\bCx)^r$-direction.

%In order to compute the Lefschetz wrapped Floer cohomologies,
%%$\HW_l(\hat{L}_i, \hat{L}_j)$,
%it is useful to introduce an intermediary object
%\begin{align*}
% \HW_l'(\hat{L}_i, \hat{L}_j) := \varinjlim_{w} \HF(\phi'_w(\hat{L}_i), \hat{L}_j),
%\end{align*}
%where $\phi'_w$ is the time $w$ flow
%defined by the Hamiltonian
%$H'(\bsz, u, v) = H_b(\bsz) + H_f'(u, v)$
%where $H_f'$ wraps `half' of the fiber
%as in \pref{fg:fiber_wrapping3}.
%Possible choices for $H_f$ and $H_f'$ are given by
%$H_f(u, v) = |u|^2 + |v|^2$ and $H_f'(u, v) = |u|^2$.

\begin{figure}[htbp]
\begin{minipage}{.5 \linewidth}
\centering
\input{wrapping1.pst}
\caption{Wrapping once}
\label{fg:wrapping1}
\end{minipage}
\begin{minipage}{.5 \linewidth}
\centering
\scalebox{.95}{
%\vspace{5mm}
\input{wrapping2.pst}
%\vspace{5mm}
}
\caption{Wrapping twice}
\label{fg:wrapping2}
\end{minipage}
%\begin{minipage}{.24 \linewidth}
%\centering
%\vspace{5mm}
%\input{fiber_wrapping3.pst}
%\vspace{6mm}
%\caption{Wrapping the fiber by $H_f'$}
%\label{fg:fiber_wrapping3}
%\end{minipage}
\end{figure}

Note that the monodromy of the conic fibration
$\varpi : E \to S$
around the discriminant $S_0$
is given by the Dehn twist
along the vanishing cycle,
which is inverse to the wrapping on the fiber.
It follows that
for a positive integer $w$,
the intersection points
$\phi_w(\hat{L}_{i_0}) \cap \hat{L}_{i_1}$
can be parametrized as
$q_j$ for some $q \in \sigma^w(\Delta_{i_0}) \cap \Delta_{i_1}$
and an integer $j \in [0, \ord(q)+w]$
indicating the position in the fiber of $\varpi$
counted from the top
as shown in Figures \ref{fg:wrapping1} and \ref{fg:wrapping2}.
The continuation map
$\phi_{w+1, w}^\scrA : \scrA_w \to \scrA_{w+1}$
is written as
$\phi_{w+1, w}^\scrA(q_j) = \frakt(q)_{j+1}$
in this parametrization.
We write $q = \frakt^{\ord (q)}(q')$, and
%One has an isomorphism
%\begin{align*}
% \HF(\phi_w (\hat{L}_{i_0}), \hat{L}_{i_1})
% &\cong
%  \Hom_{\cF(S, S_*)}(\Delta_{i_0}, \Delta_{i_1})
%   \oplus \bigoplus_{k=1}^w \Hom_{\cF(S, S_*)}(\sigma^k(\Delta_{i_0}), \Delta_{i_1})^{\oplus 2}
%\end{align*}
%of $\bC$-vector spaces,
%sending $q_j \in \phi_w(\hat{L}_{i_0}) \cap \hat{L}_{i_1}$
%%for $q \in \sigma^w(\Delta_{i_0}) \cap \Delta_{i_1}$ and $j \in [-\ord(q), \ord(q)]$
%to
%\begin{align*}
% q &\in \Hom_{\cF(E, E_*)}(\Delta_{i_0}, \Delta_{i_1}),
%  \hspace{27mm} j = 0, \\
% (\frakt^{\ord(q) - j}(q'), 0)
%  &\in \Hom_{\cF(E, E_*)}(\sigma^{\ord(q)-j}(\Delta_{i_0}), \Delta_{i_1})^{\oplus 2},
%   \qquad j > 0, \\
% (0, \frakt^{\ord(q) + j}(q'))
%  &\in \Hom_{\cF(E, E_*)}(\sigma^{\ord(q)+j}(\Delta_{i_0}), \Delta_{i_1})^{\oplus 2},
%   \qquad j < 0.
%\end{align*}
consider the $\bC$-linear map
\begin{align} \label{eq:psi}
\begin{array}{cccc}
 \psi_w : & \HF(\phi_w (\hat{L}_{i_0}), \hat{L}_{i_1})
  & \to & \bigoplus_{k=0}^{2w} \Hom_\scX(E_{i_0} \otimes \omega_\scX^k, E_{i_1}) \\
 & \vin & & \vin \\
 & q_j & \mapsto & (s-1)^j Q'
\end{array}
\end{align}
where $Q' \in \Hom_\cX(E_{i_0} \otimes \omega_\cX^{w-\ord(q)}, E_{i_1})$
corresponds to $q' \in \Hom_{\cF(S, S_*)} \lb \sigma^{w-\ord(q)}(\Delta_{i_0}), \Delta_{i_1} \rb$
under the ring isomorphism in \pref{th:ring_isom}.

Let
$$
 \scrB := \varinjlim_{w \to \infty} \scrB_w, \quad
 \scrB_w := \bigoplus_{k=0}^{2w} \bigoplus_{i,j=0}^n \Hom_\cX(E_i \otimes \omega_\cX^k, E_j)
$$
be the direct limit of the right hand side of \eqref{eq:psi}
with respect to the map
$$
 \phi_{w+1, w}^\scrB : \scrB_w \to \scrB_{w+1}
%  \bigoplus_{k=0}^{2w} \bigoplus_{i,j=0}^n \Hom_\cX(E_i \otimes \omega_\cX^k, E_j)
%   \to \bigoplus_{k=0}^{2(w+1)} \bigoplus_{i,j=0}^n \Hom_\cX(E_i \otimes \omega_\cX^k, E_j)
$$
given by the multiplication by $1-s$
$$
 \bigoplus_{k=0}^{2w} \bigoplus_{i,j=0}^n \Hom_\cX(E_i \otimes \omega_\cX^k, E_j)
  \xto{1-s} \bigoplus_{k=0}^{2w+1} \bigoplus_{i,j=0}^n \Hom_\cX(E_i \otimes \omega_\cX^k, E_j)
$$
%by $(1-s) \in H^0(\scO_\cX) \oplus H^0(\omega_\cX^{-1})$
followed by the inclusion
$$
 \bigoplus_{k=0}^{2w+1} \bigoplus_{i,j=0}^n \Hom_\cX(E_i \otimes \omega_\cX^k, E_j)
  \hookrightarrow \bigoplus_{k=0}^{2w+2} \bigoplus_{i,j=0}^n \Hom_\cX(E_i \otimes \omega_\cX^k, E_j).
$$
The multiplication in $\scrB$ is defined by
\begin{align*}
\begin{array}{ccc}
 \scrB_{w_2} \otimes \scrB_{w_1} & \to & \scrB_{w_1+w_2} \\
 \vin & & \vin \\
 Q \otimes R & \mapsto & Q \cdot R.
\end{array}
\end{align*}

\begin{proposition} \label{pr:ring_isom4}
One has an isomorphism
\begin{align}
 \scrB \simto \bigoplus_{i,j=0}^n \Hom_\cYv(\pi^* E_i|_\cYv, \pi^* E_j|_\cYv)
\end{align}
of rings.
\end{proposition}

\begin{proof}
By combining
\begin{align*}
 \bigoplus_{i,j=0}^n \Hom_{\cK}(\pi^* E_i, \pi^* E_j)
  &\cong \bigoplus_{i,j=0}^n \Hom_{\cX}(E_i, \pi_* \pi^* E_j) \\
  &\cong \bigoplus_{i,j=0}^n \Hom_{\cX}(E_i, E_j \otimes \pi_* \cO_\cK) \\
  &\cong \bigoplus_{i,j=0}^n \Hom_{\cX}
   \lb E_i, E_j \otimes \lb \bigoplus_{k=0}^\infty \omega_\cX^{-k} \rb \rb \\
  &\cong \varinjlim_{w \to \infty}
   \bigoplus_{k=0}^{w} \bigoplus_{i,j=0}^n \Hom_\cX(E_i \otimes \omega_\cX^k, E_j)
\end{align*}
with
\begin{align*}
 \Hom(M|_\cYv, N|_\cYv) = \varinjlim \lb \Hom(M, N) \xto{1-s} \Hom(M, N) \xto{1-s} \cdots \rb
\end{align*}
for any objects $M$ and $N$ of $D^b \coh \cK$
(cf. \cite[(1.13)]{Seidel_ASNT}),
one obtains an isomorphism
$$
 \bigoplus_{i,j=0}^n \Hom_{\cYv}(\pi^* E_i|_{\cYv}, \pi^* E_j|_{\cYv})
  \cong \varinjlim_{w \to \infty}
   \bigoplus_{k=0}^{2w} \bigoplus_{i,j=0}^n \Hom_\cX(E_i \otimes \omega_\cX^k, E_j)
$$
of $\bC$-vector spaces.
The multiplication in $\scrB$ commutes with this $\bC$-linear isomorphism,
and \pref{pr:ring_isom4} is proved.
\end{proof}

\begin{proposition} \label{pr:comp}
The maps $\psi_w$ are compatible with the composition,
i.e.,
\begin{align} \label{eq:comp}
 \psi_{w_1+w_2}(\frakm_2(q_j, r_k))
  = \psi_{w_2}(q_j) \cdot \psi_{w_1}(r_k)
\end{align}
for any $q_j \in \phi_{w_2}(\hat{L}_{i_1}) \cap \hat{L}_{i_2}$
and $r_k \in \phi_{w_1}(\hat{L}_{i_0}) \cap \hat{L}_{i_1}$.
\end{proposition}

\begin{proof}
Since the Lagrangian submanifolds $\hat{L}_i$ are fibered over $\Delta_i$,
any holomorphic triangle in $\varphitilde : D^2 \to E$
contributing to $\frakm_2(q_j, r_k)$ projects
to a holomorphic triangle $\varphi = \varpi \circ \varphitilde : D^2 \to S$
contributing to $\frakm_2(q, r)$.
Given a holomorphic triangle $\varphi : D^2 \to S$
contributing to $p = \frakm_2(q, r)$,
the contributions to $\frakm_2(q_j, r_k)$
of holomorphic triangles $\varphitilde : D^2 \to E$
projecting to $\varphi$ is computed by Pascaleff
\cite[Proposition 4.4]{Pascaleff_FCMPP}
as
\begin{align*}
 \frakm_2(q_j, r_k) = \sum_{t=0}^\ell \binom{\ell}{t} p_{j+k+t},
\end{align*}
where $\ell = \varphi(D^2) \cdot S_0$ is the intersection number
of the triangle and the discriminant $S_0$ of the fibration $\varpi : E \to S$.
It follows that
\begin{align*}
 \psi_{w_1+w_2}(\frakm_2(q_j, r_k))
  &= \sum_{t=0}^\ell \binom{\ell}{t} \psi_{w_1+w_2}(p_{j+k+t}) \\
  &= \sum_{t=0}^\ell \binom{\ell}{t} (s-1)^{j+k+t} P'.
\end{align*}
On the right hand side of \eqref{eq:comp},
one has
\begin{align*}
 \psi_{w_2}(q_j) \cdot \psi_{w_1}(r_k)
  &= (s-1)^j Q' \cdot (s-1)^k R' \\
  &= (s-1)^{j+k} Q' R' \\
  &= (s-1)^{j+k} s^{\ell'} P' \\
  &= \sum_{t=0}^{\ell'} \binom{\ell'}{t} (s-1)^{j+k+t} P'
\end{align*}
where $\ell' = \ord(p)-\ord(q)-\ord(r)$.
Now one has $\ell = \ell'$ by \pref{pr:degree},
and \pref{pr:comp} is proved.
\end{proof}

\begin{proposition} \label{pr:ring_isom3}
The maps $\psi_w$ induce an isomorphism
%\begin{align}
% \psihat : \bigoplus_{i,j=0}^n \HW(L_i, L_j)
%  \simto \bigoplus_{i,j=0}^n \Hom_\cYv(\pi^* E_i|_\cYv, \pi^* E_j|_\cYv)
%\end{align}
$
 \psi : \scrA \simto \scrB
$
of rings.
\end{proposition}

\begin{proof}
The maps $\psi_w$ induce a map
$
 \psi : \scrA \to \scrB
$
between inductive limits
since
\begin{align*}
 \psi_{w+1}(\phi_{w+1, w}^\scrA(q_i))
  &= \psi_{w+1}(q_{i+1}) \\
  &= (s-1)^{i+1} \cdot Q' \\
  &= \phi_{w+1, w}^\scrB((s-1)^i \cdot Q') \\
  &= \phi_{w+1, w}^\scrB( \psi_w(q_i)).
\end{align*}

For any element $q \in \Hom_\cYv(\pi^* E_i|_\cYv, \pi^* E_j|_\cYv)$,
one can take sufficiently large $w_1$
so that $(s-1)^{w_1} q$ extends to an element of
$\Hom_\cK(\pi^* E_i, \pi^* E_j)$.
Then one takes another sufficiently large $w_2$
so that $(s-1)^{w_1} q$ comes from
$
 \bigoplus_{k=0}^{w_2} \Hom_\cX(E_i \otimes \omega_\cX^k, E_j)
  \subset \Hom_\cK(\pi^* E_i, \pi^* E_j).
$
Then $q$ is in the image of
$\psi \circ \phi_{\infty, w_1+w_2}^\scrA : \scrA_{w_1+w_2} \to \scrB$,
which shows that $\psi$ is surjective.

Note that the map $\phi_{\infty, w}^\scrB : \scrB_w \to \scrB$ is injective
since each map $\phi_{w+1, w}^\scrB: \scrB_w \to \scrB_{w+1}$ is injective.
The map $\psi_w$ is also injective,
and hence the map
$
 \psi \circ \phi_{\infty, w}^\scrA = \phi_{\infty, w}^\scrB \circ \psi_w : \scrA_w \to \scrB
$
is injective. This implies the injectivity of $\psi$, and
\pref{pr:ring_isom3} is proved.
\end{proof}

Now we prove \pref{th:main2}.

\begin{proof}[Proof of \pref{th:main2}]
Recall that an object $\cE$ in a triangulated category is {\em acyclic}
if $\End^*(\cE)$ is concentrated in degree zero.
It is a {\em generator} if $\Hom(\cE, X) = 0$ implies $X \cong 0$.
An acyclic generator is called a {\em tilting object}.
Since the exceptional collection $(E_i)_{i=0}^n$ is full,
the pull-back $\bigoplus_{i=0}^n \pi^* E_i$ is a generator of $D^b \coh \cK$.
Cyclicity of $(E_i)_{i=0}^n$ implies the acyclicity of $\bigoplus_{i=0}^n \pi^* E_i$.
This shows that $\bigoplus_{i=0}^n \pi^* E_i$ is a tilting object.
It follows that the restriction $\bigoplus_{i=0}^n \pi^* E_i|_\cYv$ is a tilting object.
(In general, the restriction of a generator
to an open subset is a generator,
and the restriction of an acyclic object
to the complement of a principal divisor is acyclic.)
Morita theory for derived categories
\cite{Rickard,Bondal_RAACS}
shows that $D^b \coh \cYv$ is equivalent
to the derived category of finitely-generated modules
over $\End \lb \bigoplus_{i=0}^n \pi^* E_i|_\cYv \rb$.

The direct sum $\bigoplus_i \hat{L}_i$ is acyclic by \pref{pr:ring_isom3},
and it is a generator of $\cW'$ by definition.
It follows that $D^b \cW'$ is equivalent to
to the derived category of finitely-generated modules
over $\End \lb \bigoplus_{i=0}^n \hat{L}_i \rb$,
which is isomorphic to $\End \lb \bigoplus_{i=0}^n \pi^* E_i|_\cYv \rb$
by \pref{pr:ring_isom3}.
This concludes the proof of \pref{th:main2}.
\end{proof}

\begin{figure}
%\begin{minipage}{.5 \linewidth}
\centering
\input{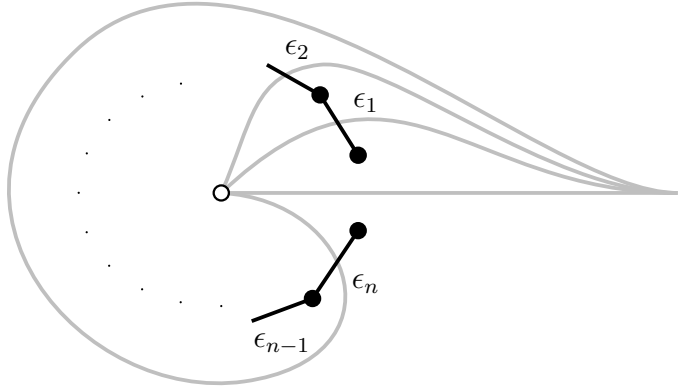}
\caption{The dual cycles}
\label{fg:thimbles7}
%\end{minipage}
\end{figure}

%The compatibility between Theorem \ref{th:main1}
%and Theorem \ref{th:main2} is given as follows:
Finally we discuss the compatibility
of Theorem \ref{th:main1} and Theorem \ref{th:main2}.
Let $(\epsilon_1, \ldots, \epsilon_n)$
be the collection of line segments on $E$
connecting points in $W^{-1}(0)$ as in Figure \ref{fg:thimbles7},
so that the intersection numbers with Lefschetz thimbles are given by
$$
 \Delta_i \cdot \epsilon_j = \delta_{ij},
  \quad i = 0, \ldots, n, \ j = 1, \ldots, n.
$$

One can choose a symplectomorphism $Y \simto Y'$
in such a way that these intersection numbers
correspond precisely to the winding numbers
in Definition \ref{df:winding_number}.
The SYZ transformations of the resulting Lagrangians in $Y$
are then given by line bundles $(\scL_i)_{i=0}^n$ on $\Yv$
satisfying
$$
 \deg \scL_i|_{E_j} = \delta_{ij}
$$
for $i = 0, \ldots, n$ and $j = 1, \ldots, n$.
One can easily see,
either from \cite[Section 2]{Kapranov-Vasserot}
or by a direct calculation,
that the endomorphism rings of $\bigoplus_{i=0}^n \pi^* E_i|_\cYv$
and $\bigoplus_{i=0}^n \cL_i$ are isomorphic,
so that there is an equivalence $D^b \coh \scYv \simto D^b \coh \Yv$
sending $\pi^* E_i|_\cYv$ to $\scL_i$.
This gives the compatibility
of Theorem \ref{th:main1} and Theorem \ref{th:main2}
discussed in Introduction.

\bibliographystyle{amsalpha}
\bibliography{bibs}

\noindent
Kwokwai Chan

Department of Mathematics,
The Chinese University of Hong Kong,
Shatin,
Hong Kong

{\em e-mail address}\ :\ kwchan@math.cuhk.edu.hk
\ \vspace{0mm} \\

\noindent
Kazushi Ueda

Department of Mathematics,
Graduate School of Science,
Osaka University,
Machikaneyama 1-1,
Toyonaka,
Osaka,
560-0043,
Japan.

{\em e-mail address}\ : \  kazushi@math.sci.osaka-u.ac.jp
\ \vspace{0mm} \\

\end{document}